\documentclass[10pt,reqno]{amsart}

\usepackage{amsmath}
\usepackage{amsthm}
\usepackage{amssymb}
\usepackage{amsfonts}
\usepackage{subcaption}
\usepackage{graphicx}
\usepackage{latexsym}
\usepackage{cite, url}
\usepackage[dvipsnames]{xcolor}
\usepackage{hyperref}

\usepackage{mathtools}
\mathtoolsset{showonlyrefs}

\usepackage{bm} 

\usepackage{enumitem}
    \setenumerate{itemsep=5pt} 
    \setitemize{itemsep=5pt} 
 \setlist[enumerate]{leftmargin=*}
  \setlist[itemize]{leftmargin=*}    
\usepackage{mathtools}

\usepackage[sc]{mathpazo}
\usepackage[T1]{fontenc}

\graphicspath{{IMAGES/}}

\usepackage[font=Large]{caption}

\usepackage[font=footnotesize,labelfont=bf]{caption}

\newcommand{\bfell}{\bm{\ell}}
\newcommand{\h}{\mathcal{H}}
\newcommand{\E}{\mathcal{E}}

\newcommand{\K}{\mathcal{K}}

\renewcommand{\S}{\mathcal{S}}

\newcommand{\Z}{\mathbb{Z}}
\newcommand{\R}{\mathbb{R}}
\newcommand{\C}{\mathbb{C}}
\newcommand{\M}{\mathsf{M}}

\newcommand{\T}{\mathsf{T}}

\newcommand{\QQ}{\mathcal{Q}}

\newcommand{\Lip}{\operatorname{Lip}}
\renewcommand\vec{\mathbf}

\def\spn{\operatorname{span}}
\def\dim{\operatorname{dim}}

\DeclareMathOperator{\vol}{Vol}

\newcommand{\diam}{\operatorname{diam}}

\renewcommand{\pmod}[1]{\,(\operatorname{mod} #1)}
\renewcommand{\phi}{\varphi}
\newcommand{\conv}{\operatorname{conv}}

\DeclarePairedDelimiter{\floor}{\lfloor}{\rfloor}
\DeclarePairedDelimiter{\ceil}{\lceil}{\rceil}

\newcommand{\norm}[1]{\left\|#1\right\|}

\newcommand{\rank}{\operatorname{rank}}

\newcommand{\ZZ}{\mathsf{Z}}

\newcommand{\twovectorsmall}[2]{\big[\begin{smallmatrix} #1 \\ #2  \end{smallmatrix}\big]}

\newcommand{\tworowvector}[2]{[#1\,\,#2]}

\newcommand{\fourrowvector}[4]{[ #1 \,\, #2 \,\, #3 \,\, #4]}

\newcommand{\semigroup}[1]{\langle #1 \rangle}

\newcommand{\abs}[1]{\left|#1\right|}

\newtheorem{thmx}{Theorem}
\newtheorem{thmintro}{Theorem}

\theoremstyle{plain}
\newtheorem{theorem}[equation]{Theorem}
\newtheorem{lemma}[equation]{Lemma}

\newtheorem{proposition}[equation]{Proposition}
\newtheorem{corollary}[equation]{Corollary}
\theoremstyle{definition}
\newtheorem{remark}[equation]{Remark}
\newtheorem{definition}[equation]{Definition}
\newtheorem{example}[equation]{Example}

\author{Stephan Ramon Garcia}
\address{Department of Mathematics and Statistics, Pomona College, 610 N. College Ave., Claremont, CA 91711} 
\email{stephan.garcia@pomona.edu}
\urladdr{\url{https://stephangarcia.sites.pomona.edu/}}

\author{Gabe Udell}
\address{Department of Mathematics, Cornell University, 301 Tower Rd, Ithaca, NY 14853}
\email{gru5@cornell.edu}
\urladdr{\url{https://e.math.cornell.edu/people/Gabe_Udell/}}

\allowdisplaybreaks
\begin{document}
\title[Factorization length distribution for affine semigroups V]{Factorization length distribution for affine semigroups~V: explicit asymptotic behavior of weighted factorization lengths on numerical semigroups}

\begin{abstract}
We describe the asymptotic behavior of weighted factorization lengths on numerical semigroups.  Our approach is geometric as opposed to analytic, explains the presence of Curry--Schoenberg B-splines as limiting distributions, and provides explicit error bounds (no implied constants left unspecified). Along the way, we explicitly bound the difference between the vector partition function and the multivariate truncated power associated to a $2 \times k$ matrix.
\end{abstract}



\maketitle

\section{Introduction}

A \emph{numerical semigroup} is an additive subsemigroup of $\Z_{\geq 0}$ with finite complement \cite{Assi, Rosales}.  
A numerical semigroup $S$ has a unique minimal system of \emph{generators}
$n_1 < n_2 < \cdots < n_k$ in $\Z_{>0}$ such that $\gcd(n_1,n_2,\ldots,n_k) = 1$ and
\begin{equation*}
S = \semigroup{n_1,n_2,\ldots,n_k}
= \{ x_1 n_1 + x_2 n_2 + \cdots + x_k n_k : x_1,x_2,\ldots,x_k \in \Z_{\geq 0}\},
\end{equation*}
the additive subsemigroup of $\Z_{\geq 0}$ generated by $n_1,n_2,\ldots,n_k$ \cite[Thm.~2.7]{Rosales}.
We denote this minimal system of generators as a column vector $\vec{n} = [n_i]\in \Z_{>0}^k$.  For typographical 
convenience, we may write vectors as tuples like $(n_1,n_2,\ldots,n_k)$.

A \emph{factorization} of $n \in S$ is an $\vec{x} = [x_i] \in \Z^k_{\geq 0}$
such that $n= \vec{n} \cdot \vec{x}$.
The set of all factorizations of $n$ in $S$ is denoted $\ZZ_S(n)$.
The \emph{length} of a factorization $\vec{x} \in \ZZ_S(n)$ is 
$x_1+x_2+\cdots+x_k = \vec{1} \cdot \vec{x}$, in which $\vec{1}\in \R^k$ is the all-ones vector.
For each $\vec{m} \in \R^k$, we may consider the corresponding \emph{weighted factorization length} 
$\vec{m} \cdot \vec{x}$ \cite{semigroupsIV, parmFam}.

The goal of this paper, which continues the series \cite{GOY, semigroupsII, Modular, semigroupsIV}, is to understand the asymptotic distribution of weighted factorization lengths as $n$ increases. One of the main results of \cite{semigroupsII} relates the asymptotic behavior of (unweighted) factorization lengths
to the integral of a Curry--Schoenberg B-spline: 
\begin{equation}\label{eq:OldThing}
\lim_{n\to\infty} \frac{ |\{ \vec{x} \in \ZZ_S(n) : \vec{x} \cdot \vec{1} \in [ \alpha n, \beta n]\}| }{ | \ZZ_S(n)| }
= \int_{\alpha}^{\beta}  M(t; \tfrac{1}{n_k}, \tfrac{1}{n_{k-1}},\ldots, \tfrac{1}{n_1})\, dt;
\end{equation}
see Subsection \ref{Subsection:BSpline} for the definition of B-splines.  The relevant B-spline (the integrand at right above) is a piecewise-polynomial probability density supported on $[1/n_k, 1/n_1]$.  Since B-splines are fundamental in computer-aided design and implemented in most computer algebra systems, they are well understood.

The proof of \eqref{eq:OldThing} and several related results in \cite{semigroupsII} used tools from fields not typically associated with numerical semigroups, such as harmonic analysis, measure theory, and algebraic combinatorics.  Since the proof employed a hidden compactness argument (via L\'evy's continuity theorem) the rate of convergence was not  established. We rectify that issue here, and extend the results in several key directions.  For example, we can treat weighted factorization lengths and repeated generators, and we also obtain explicit error terms (that is, no implied constants are left unspecified). Our approach is geometric as opposed to functional-analytic, in contrast to \cite{semigroupsII}.  Even those familiar with splines and lattice-point counting may find the explicit nature of the error estimates appealing.

Our results are largely framed in the language of factorizations in numerical semigroups but may be of interest to a wider audience. Since semigroup factorizations are vectors with non-negative entries, the factorization length of a vector $\vec{x}$ is the same as its $1$-norm $|x_1|+|x_2|+\cdots+ |x_k|$. In this light, our results can be viewed as describing the distribution of the 1-norm over the integer points of simplices given by $\vec{n} \cdot \vec{x} = n$ with $\vec{x}\in \R_{\geq 0}^k$, where $\vec{n} \in \Z_{>0}^k$. Our methods are more closely aligned with the perspective that we are studying the distibution of (nearly) arbitrary linear functionals on the same set of simplices. These perspectives suggest the possibility of studying the distribution of the $1$-norm and of arbitrary linear functionals on more general families of polytopes. 

\medskip\noindent\textbf{Main results.}
The statements of our main theorems are technical, so we provide them here 
for the reader's convenience.  They will be repeated later in proper context, 
once the preliminaries are in place and we are ready to prove our results.

We permit numerical semigroups with non-minimal generating sets or with repeated generators, which we distinguish with colors. For example, although $\semigroup{2, 2, 3, 20} = \semigroup{ 2,3}$ as semigroups, we regard $\semigroup{2, 2, 3, 20}$ as having a ``red $2$'' and a ``blue $2$'' as generators, so their factorization sets are different.  
We can consider weighted factorization lengths in this expanded setting without issue.

The first main result (Theorem \ref{Theorem:A}) gives an explicit estimate 
for the number of factorizations of $n$ such that the scaled weighted factorization length 
$\vec{m} \cdot (\vec{x}/n)$ has values in $[ \alpha , \beta]$.
This permits a broad generalization of the asymptotic statements from \cite{semigroupsII}.
In particular, we are now able to estimate essentially arbitrary statistics (such as the mean, median, mode, moments, variance, standard deviation, skewness, kurtosis, harmonic and geometric means, and so forth) for weighted factorization lengths as $n \to \infty$.

\begin{thmintro}
Let $S = \langle n_1,n_2,\ldots,n_k \rangle \subseteq \Z_{\geq 0}$ with 
\begin{equation*}
n_1 \leq n_2\leq \cdots \leq n_k
\quad \text{and} \quad \gcd(n_1, n_2, \dots,n_k)=1. 
\end{equation*}
If $\vec{m}= [m_i]  \in \R^k$ and $\vec{n} = [n_i] \in \Z_{>0}^k$ are linearly independent, then
for all $n\in \Z_{>0}$ and $\alpha,\beta\in \R \cup \{\pm \infty\}$ with $\alpha \leq  \beta$,
{\small
\begin{equation*}
\left|    \Big|\!\{\vec x\in \ZZ_S(n): \vec{m} \cdot (\vec{x} / n ) \in [\alpha , \beta ]\} \!\Big|  
- 
n^{k-1}\!\! \int_{\alpha}^{\beta} \frac{M\big(t; \frac{m_1}{n_1}, \frac{m_2}{n_2}, \ldots, \frac{m_k}{n_k} \big)}{(k-1)!n_1n_2\cdots n_k}\,dt \right| \leq n^{k-2} E_1(k),
\end{equation*}
}%
in which $M\big(t; \frac{m_1}{n_1}, \frac{m_2}{n_2}, \ldots, \frac{m_k}{n_k} \big)$ is a Curry--Schoenberg B-spline and
\begin{equation*}
E_1(k)=  8^{k-2} (k-1)^{\frac{3k^2-k-7}{2}}.
\end{equation*}
If $f:\R\to\R$ is bounded and continuous, then
\begin{equation}\label{eq:ThmABegin}
\lim_{n\to\infty} \frac{1}{|\ZZ_S(n)|} \sum_{ \vec{x} \in \ZZ_S(n)} f \Big( \frac{\vec{m} \cdot \vec{x} }{n} \Big) =
 \int_{-\infty}^{\infty} \, f(t) M\Big(t; \frac{m_1}{n_1}, \frac{m_2}{n_2}, \ldots, \frac{m_k}{n_k} \Big)\,dt.
\end{equation}
\end{thmintro}

Before proceeding, we emphasize again that the bound in Theorem \ref{Theorem:A} is explicit and contains no unspecified constants. Theorem \ref{Theorem:A} improves over \cite{semigroupsII} immensely: arbitrary weights and repeated, non-minimal generators can be treated.  

The second main result (Theorem \ref{Theorem:B}) bounds the difference between 
the vector partition function $t_A(\vec{b})$, which
counts the number of integer points in the variable polytope 
$\Pi_A(\vec{b}) = \{\vec{x}\in \R_{\geq 0}^k: A\vec{x} = \vec{b}\}$, and
the multivariate truncated power $T_A(\vec{b})$, that is, the volume of $\Pi_{A}(\vec{b})$, where $A = \tworowvector{ \vec{m}}{\vec{n}}^{\T} \in \M_{2\times k}(\Z)$. Thus, one can approximate one quantity with the other with an explicit error bound. Partition functions come up throughout mathematics \cite{goodRepLit, litonkonstant, StanleyCCA,fancypartitions,matroids,MR1380519}; in our context, $t_A([m,n]^\T)$ is the number of factorizations of $n$ with weighted factorization length $m$. We state Theorem \ref{Theorem:B} on its own since such estimates may appeal 
to those who work with vector partition functions or multivariate truncated powers and have applications in other contexts. While Theorem \ref{Theorem:A} describes the distribution of weighted factorization lengths, Theorem \ref{Theorem:B} describes the approximate number of factorizations of $n$ with a specified weighted factorization length.

\begin{thmintro}
Let $\vec{n} = [n_i] \in \Z_{>0}^k$ and $\vec{m} =[m_i] \in \Z^k$, and suppose $A \Z^k = \Z^2$, in which $A = \tworowvector{ \vec{m}}{\vec{n}}^{\T} \in \M_{2\times k}(\Z)$.
For each $\vec{b} \in \Z^2$,
\begin{equation*}
 |t_A(\vec{b})-T_A(\vec{b}) | \leq n^{k-3}E_2(k),
\end{equation*}
in which
\begin{equation*}
E_2(k)
= 8^{k-3} (k-2)^{\frac{3k^2-7k-3}{2}}.
\end{equation*}
\end{thmintro}

The order of magnitude (in $n$) of the error term should not surprise readers versed in Ehrhart theory.  However, as far as the authors are aware, such explicit bounds have not been previous published.
A concrete approximation is worthwhile because even the problem of computing $t_{\vec{n}^{\T}}(n) = |\ZZ_S(n)|$
is \#P-hard. 
For our purposes, Theorem \ref{Theorem:B} is a major ingredient in the proof of Theorem \ref{Theorem:C} below.

Our last main result is a version of the asymptotic statement \eqref{eq:ThmABegin},
this time with explicit error bounds and the opportunity to localize
to a given interval. 
Since Theorem \ref{Theorem:C} is proved from Theorem \ref{Theorem:B}, it inherits its more restrictive hypotheses.

\begin{thmintro}
Let $\vec{m} =[m_i] \in \Z^k$ and $\vec{n} = [n_i] \in \Z_{>0}^k$, and let $S = \semigroup{n_1,n_2,\ldots,n_k}$.
\begin{enumerate}\addtolength{\itemsep}{5pt}
\item Suppose $A \Z^k = \Z^2$, in which $A = \tworowvector{ \vec{m}}{\vec{n}}^{\T} \in \M_{2\times k}(\Z)$.
\item Let $n\geq 1$ and $\alpha,\beta\in \R \cup \{\pm \infty\}$ such that $\frac{1}{n} \leq \beta - \alpha$.
\item Let $f:[\alpha, \beta] \to \R$ be continuous with $|f(x)| \leq C_1$ for $x\in [\alpha, \beta]$.
\item For $x\in [\alpha n,\beta n]$, suppose $|f(x)T_A\big( \twovectorsmall{x}{1} \big)| \leq C_2$ and has Lipschitz constant $C_3$.
\end{enumerate}
Then, 
\begin{equation*}
\left| \sum_{\substack{\vec{x}\in \ZZ_S(n) \\ \vec{m} \cdot (\vec{x}/n) \in [\alpha,\beta]}}f\left(\frac{\vec{m} \cdot \vec{x}}{n}\right) 
- 
n^{k-1}\int_{\alpha}^\beta f(t) \frac{M\big(m;\frac{m_1}{n_1},\frac{m_2}{n_2},\ldots, \frac{m_k}{n_k}\big)}{(k-1)!n_1n_2\dots n_k}  \,dt \right| 
\leq E_{3}(n,k),
\end{equation*}
in which $M\big(t; \frac{m_1}{n_1}, \frac{m_2}{n_2}, \ldots, \frac{m_k}{n_k} \big)$ is a Curry--Schoenberg B-spline and
\begin{equation*}
\begin{split}
E_{3}(n,k)
&=
\left((\beta-\alpha)\left(C_1 8^{k-3} (k-2)^{\frac{3k^2-7k-3}{2}}+C_3\right)+2C_2\right)n^{k-2} \\
&\qquad + C_1 8^{k-3} (k-2)^{\frac{3k^2-7k-3}{2}}n^{k-3}.
\end{split}
\end{equation*}
\end{thmintro}

Hypothesis (d) poses no difficulties since $T_A(\cdot)$ can be expressed explicitly in terms of
B-splines; see Theorem \ref{eq:TM}.  We have stated it as we have for the sake of typography.
Since B-splines are piecewise polynomial, an admissible Lipschitz constant in (d) can be obtained by 
a standard mean-value theorem argument.

\medskip\noindent\textbf{Acknowledgements.}
We thank Lenny Fukshansky for valuable references on lattice theory,
Claudio Procesi for answering questions about his book \cite{hyperplane}, and Amos Ron for answering questions about the connection between B-splines and multivariate truncated powers. The first listed author thanks Christopher O'Neill and Jes\'us de Loera for helpful comments. The second listed author also thanks Hunter Stufflebeam, Eyob Tsegaye, and Bilge Koksal for helpful conversations about analytical aspects of this paper and its previous drafts. He is grateful for helpful conversations with Allen Knutson, Karola M\'esz\'aros, Vasu Tewari, Alejandro H. Morales, and Pamela E. Harris. He frequently discussed the project with friends and colleagues and would like to thank Anna Barth, Aria Beaupre, Henry Bosch, Max Chao-Haft, Andrew Chen, Raj Gandhi, Jake Hauser, Robin Huang, Brett Hungar, Emma Jin,  Cole Kurashige, David Miller, Amina Mohamed,  Aliyah Newell,  Luis Perez,    Mark Schachner,  Kye Shi,  Kamryn Spinelli, and Tim Wesley for editorial and emotional support.

\medskip\noindent\textbf{Organization.}
This paper is organized as follows.  First, Section \ref{Section:Preliminaries} contains preliminary material about B-splines, 
multivariate truncated powers, vector partition functions, lattices, volume computations, and lattice-point counting.
The statement and proof of Theorem \ref{Theorem:A} are in Section \ref{Section:ProofSemigroup}.
Then in Section \ref{Section:TT} we present Theorem \ref{Theorem:B} and its proof.  Section \ref{Section:TheoremLL}
is devoted to the statement and proof of Theorem \ref{Theorem:C}. We present a wide array of examples and applications of our results in Section \ref{Section:Examples} before wrapping up in Section \ref{Section:Further} with suggestions for further research.

\section{Preliminaries}\label{Section:Preliminaries}
This section contains the preliminary material necessary for our main results.
Subsection \ref{Subsection:BSpline} introduces B-splines and their properties.
We discuss multivariate truncated powers and vector partition functions in Subsections \ref{Subsection:Truncated} and
Subsection \ref{Subsection:PartitionFunction}, respectively.
Next, Subsection \ref{Subsection:Lattices} concerns lattices.
We wrap things up in Subsection \ref{Subsection:Volume} with volume computations and lattice-point counting.

\subsection{B-splines}\label{Subsection:BSpline}
A \emph{spline} is a piecewise-polynomial function.
They appear in fields as diverse as numerical analysis \cite{Numan}, computer-aided design \cite{interpSplines,Butzer}, statistics and machine learning \cite{stats}, and tomography \cite{Butzer}. We require the Curry--Schoenberg B-spline and the multivariate truncated power (Subsection \ref{Subsection:Truncated}). 

The ``B'' in ``B-spline'' stands for `basis' because the vector space of piecewise-polynomial functions on $\R$ with specified degree, breakpoints, and smoothness at the breakpoints has a basis comprised of B-splines \cite[Thm.~44]{practicalGuide}. 
For a sequence of \emph{knots} $a_1\leq a_2 \leq \cdots \leq a_n$, not all equal, the Curry--Schoenberg \emph{B-spline} is 
\begin{equation*}
M(x;a_1,a_2,\ldots,a_n)=(n-1) M_{1,n-1}(x),
\end{equation*}
where $M_{i,k}(x)$ is defined recursively by
\begin{equation}\label{BsplineBaseCase}
M_{i,1}(x)=
\begin{cases} 
(a_{i+1}-a_i)^{-1} & \text{if $a_i\leq x <  a_{i+1}$},\\ 
0 & \text{otherwise},
\end{cases}
\end{equation}
with
\begin{equation}\label{BsplineRecursion}
M_{i,k}(x)=\frac{x-a_i}{a_{i+k}-a_i}M_{i,k-1}(x)+\frac{a_{i+k}-x}{a_{i+k}-a_i}M_{i+1,k-1}(x) 
\end{equation}
as shown in \cite[eq.~(8)]{deBoorRecursion}.
Note that
\begin{equation}\label{M to M}
    M_{i,k}(x)=\frac{1}{k}M(x;a_i,a_{i+1},\dots, a_{i+k}),
\end{equation} 
If $a_{i+1}=a_i$, then $M_{i,1}(x)=0$.
The Curry--Schoenberg B-spline $M(x;a_1,a_2,\ldots,a_n)$ is a probability density supported on $[a_1,a_n]$ that is positive on $(a_1,a_n)$ \cite[Thm.~1, eq.~(1.6)]{curry1965polya}.
B-splines are implemented in most computer algebra systems; see Figure \ref{Figure:SplineDefinition} and Remark \ref{Remark:Mathematica}.
We provide many examples in Section \ref{Section:Examples} in the context of numerical semigroups.

\begin{figure}
\centering
\begin{subfigure}{0.475\textwidth}
  \centering
  \includegraphics[width=\textwidth]{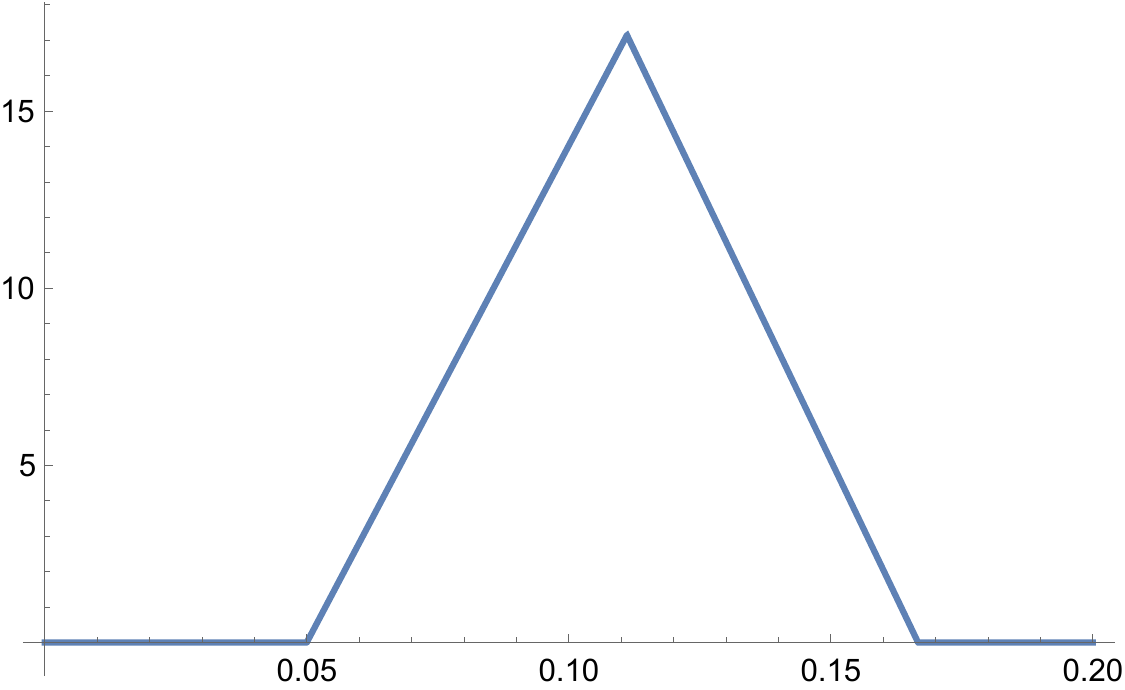}
  \caption{$M(x;\frac{1}{20}, \frac{1}{9}, \frac{1}{6})$}
\end{subfigure}
\hfill
\begin{subfigure}{0.475\textwidth}
  \centering
  \includegraphics[width=\textwidth]{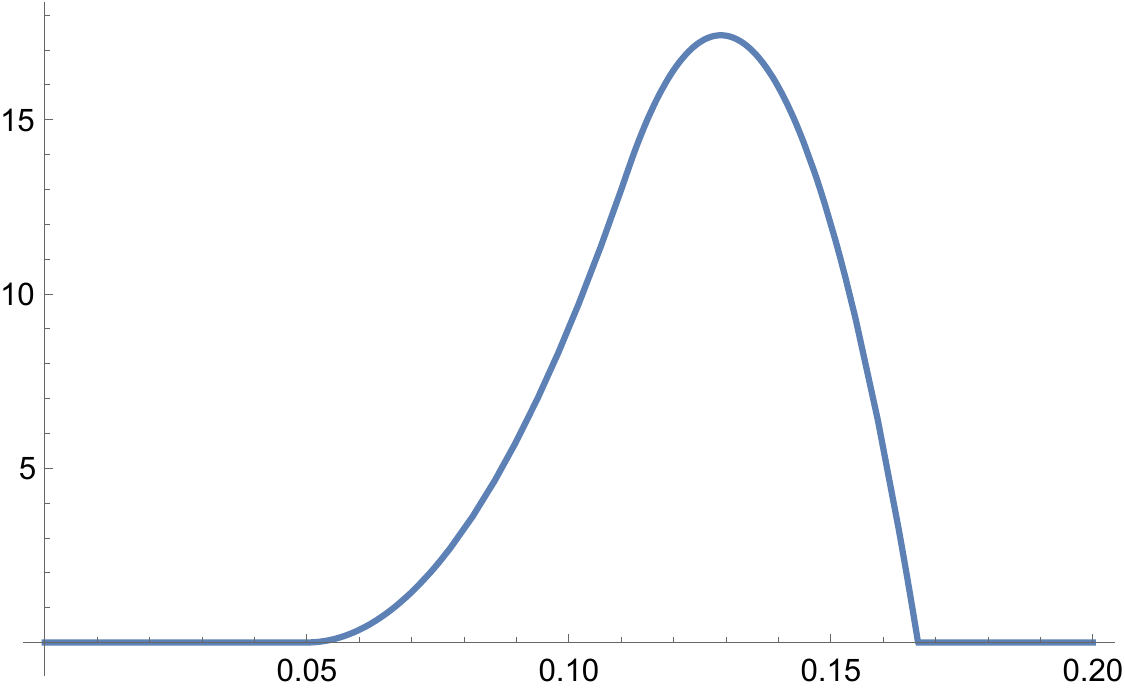}
  \caption{$M(x;\frac{1}{20}, \frac{1}{9}, \frac{1}{6},\frac{1}{6})$}
\end{subfigure}
\\
\begin{subfigure}{0.475\textwidth}
  \centering
  \includegraphics[width=\textwidth]{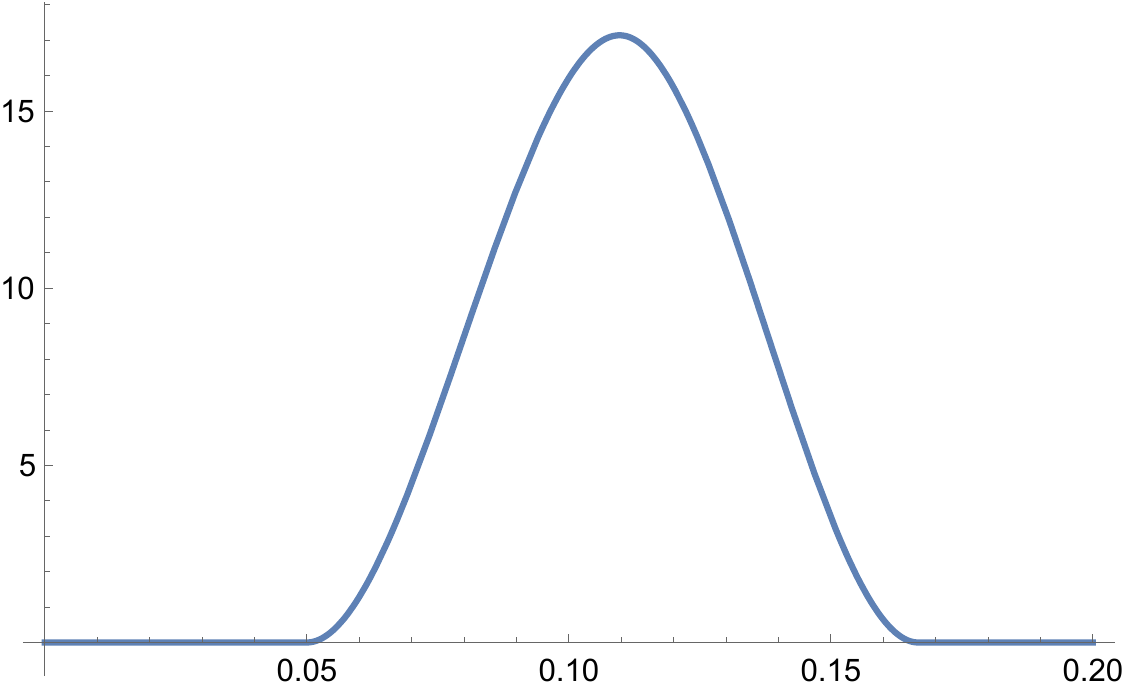}
  \caption{$M(x;\frac{1}{20},\frac{1}{20}, \frac{1}{9}, \frac{1}{6}, \frac{1}{6})$}
\end{subfigure}
\hfill
\begin{subfigure}{0.475\textwidth}
  \centering
  \includegraphics[width=\textwidth]{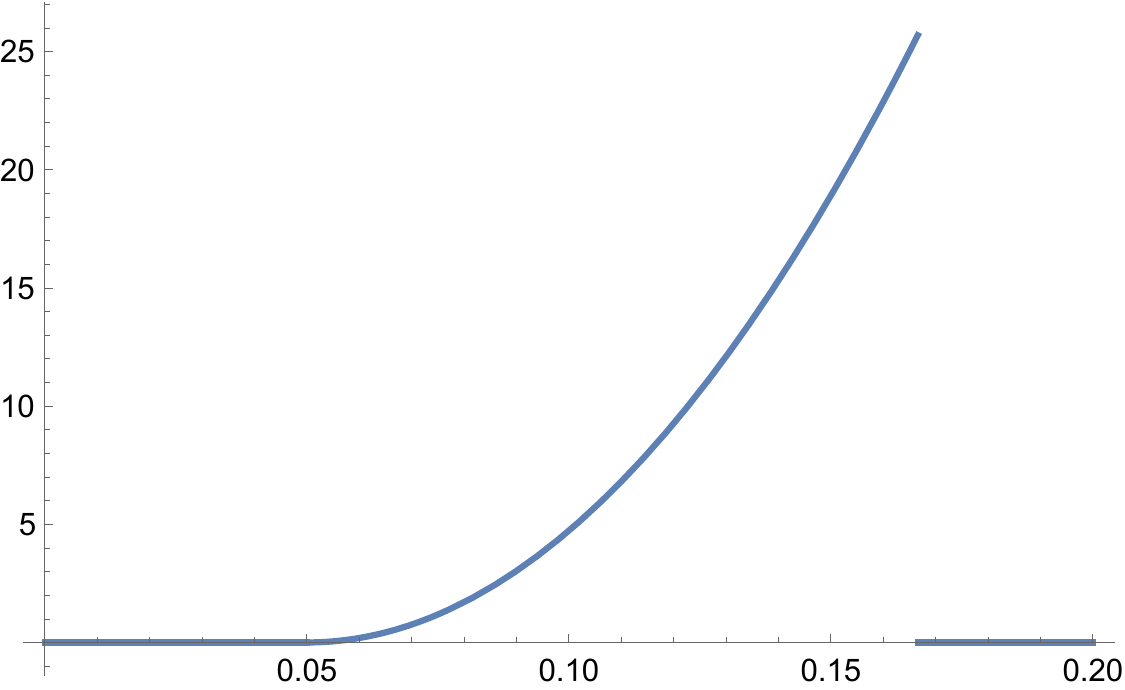}
  \caption{$M(x;\frac{1}{20}, \frac{1}{6}, \frac{1}{6},\frac{1}{6})$}
\end{subfigure}
  \\
\begin{subfigure}{0.475\textwidth}
  \centering
  \includegraphics[width=\textwidth]{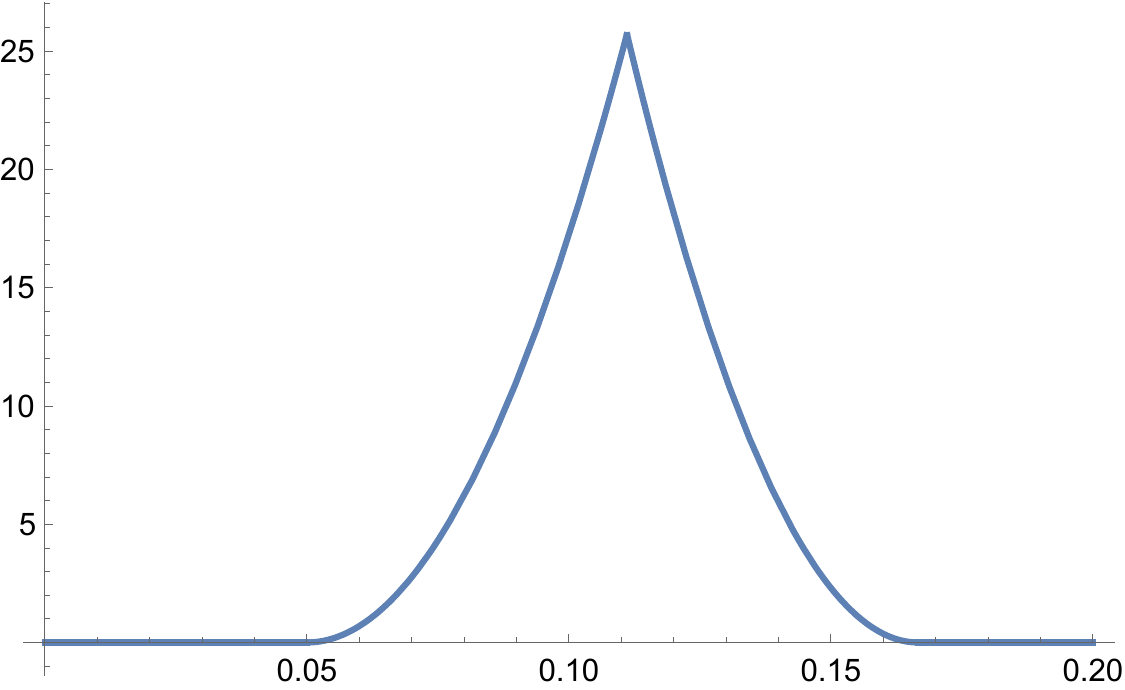}
  \caption{$M(x;\frac{1}{20},\frac{1}{9}, \frac{1}{9}, \frac{1}{6})$}
\end{subfigure}
\hfill
\begin{subfigure}{0.475\textwidth}
  \centering
  \includegraphics[width=\textwidth]{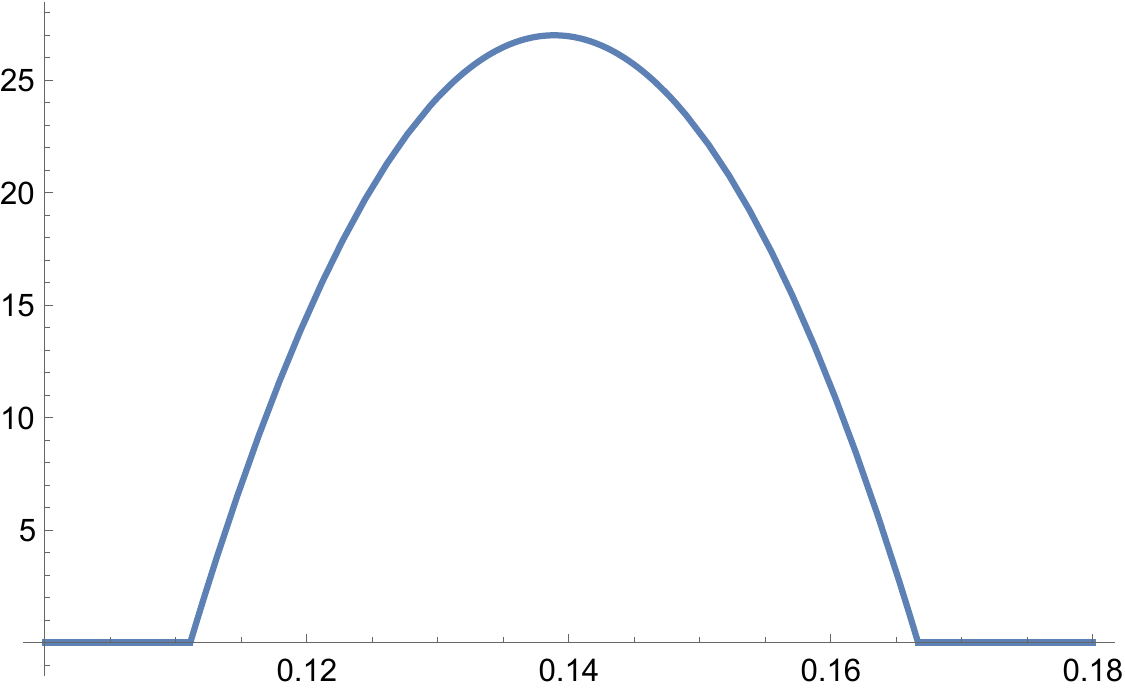}
  \caption{$M(x;\frac{1}{9}, \frac{1}{9}, \frac{1}{6},\frac{1}{6})$}
  \end{subfigure}
\caption{Plots of B-splines for various parameters related to the standard McNugget semigroup $\semigroup{6,9,20}$ (the reciprocals of the generators appear as arguments of the B-spline in Theorems \ref{Theorem:A} and \ref{Theorem:C}).}
\label{Figure:SplineDefinition}
\end{figure}

For distinct knots $a_1<a_2<\cdots< a_n$, the Curry--Schoenberg B-spline is
\begin{equation}\label{eq:ExplicitSpline}
    M(x;a_1,a_2,\dots, a_n)=(n-1)\sum_{i=1}^n \frac{(a_i-x)_+^{n-2}}{\prod_{j\neq i}(a_i-a_j)},
\end{equation}
in which $t_+=\max\{t,0\}$ \cite[eq.~(1.4)]{curry1965polya}. In this case, the B-spline is polynomial of degree $n-2$ for $x\in (a_i,a_{i+1})$ and 
is $n-3$ times differentiable at $x=a_i$. 

If the knots are not distinct, the B-spline is still polynomial of degree $n-2$ between knots but is of continuity class $C^m$ at $x=a_i$, where $m={n-2- | \{j:a_i=a_j\} |}$  \cite[Lem.~1]{curry1965polya}.  Given knots $a_1\leq a_2\leq \cdots \leq a_n$, not all equal, the corresponding B-spline is the unique function, up to scaling, which is polynomial of degree $n-2$ between knots, is $n-2-|\{j:a_i=a_j\}|$ times differentiable at $x=a_i$, and is supported on $[a_1,a_n]$ \cite[p.~9]{Handbook}.

\begin{remark}\label{Remark:Unordered}
Although the assumption $a_1\leq a_2 \leq \cdots \leq a_n$ is standard in the study of B-splines,
for convenience we permit knots to be unordered.  If $a_1,a_2,\ldots,a_n \in \R$ are not all identical, we define
\begin{equation}\label{eq:SplineUnordered}
M(x;a_1,a_2,\ldots,a_n) = M(x;a_{\pi(1)},a_{\pi(2)},\ldots,a_{\pi(n)}),
\end{equation}
in which $\pi$ is a permutation of $\{1,2,\ldots,n\}$ such that $a_{\pi(1)} \leq a_{\pi(2)} \leq \cdots \leq a_{\pi(n)}$.
\end{remark}

\begin{remark}\label{Remark:Mathematica}
B-splines are built into many computer software packages, although sometimes with different normalizations.
The images in Figure \ref{Figure:SplineDefinition} were produced with \texttt{Mathematica} via 
$\mathtt{BSplineBasis}[\{k-2, \{ a_1,a_2,\ldots,a_k\}\},0,x]$, which is then normalized to yield a probability density.
To obtain the explicit piecewise-polynomial representation of a B-spline,
use the \texttt{PiecewiseExpand} command.
\end{remark}

\subsection{Multivariate truncated powers}\label{Subsection:Truncated}
Let $A = [ \vec{a}_1 \,\, \vec{a}_2 \,\, \ldots \,\, \vec{a}_k] \in \M_{s \times k}(\R)$ with $k \geq s = \rank A$.
Suppose that $\vec{0} \notin \conv\{\vec{a}_1,\vec{a}_2, \ldots, \vec{a}_k\}$; this condition ensures that the variable polytope \eqref{eq:VariablePolytope} below is bounded. The corresponding \emph{multivariate truncated power} $T_A:\R^s\to \R$ is supported and continuous on the cone $A \R_{\geq 0}^k$,
positive on its interior \cite[Prop.~7.19]{hyperplane}, and satisfies
\begin{equation}\label{eq:DefMultivariate}
\int_{\R^s}f(\vec{b})T_A(\vec{b})\,d\vec{b} =  \int_{\R_{\geq 0}^{k} } f(A\vec{z} )\,d \vec{z}
\end{equation}
for all continuous $f:\R^s\to\R$ with compact support \cite[(7.2)]{hyperplane}.  If $\rank A < s$, one must consider $T_A$ in the sense of distributions.  However, this situation does not concern us because our standing assumption throughout this paper is that $A$ has full row rank. When $A$ has full row rank, $T_A$ is defined as the unique function which is continuous on its support and satisfies \eqref{eq:DefMultivariate}.

With $A \in \M_{s\times k}(\R)$ as above, $\rank (AA^{\T}) = s$ and $\det(AA^{\T}) > 0$.
The corresponding $(k-s)$-dimensional \emph{variable polytope} \cite[Sec.~1.3.1, p.~17]{hyperplane} is
\begin{equation}\label{eq:VariablePolytope}
\Pi_A(\vec{b}) = \{\vec{x}\in \R_{\geq 0}^k: A\vec{x} = \vec{b}\}.
\end{equation}
The following formula is \cite[Thm.~7.3]{hyperplane}:
\begin{equation}\label{eq:SplineVol}
T_A(\vec{b}) =  \frac{  \vol_{k-s} \Pi_A(\vec{b}) }{ \sqrt{\det(AA^{\T})}} ,
\end{equation}
in which $\vol_d(\cdot)$ denotes the $d$-dimensional volume of the indicated set.

\begin{remark}
It is hard to find a precise, universal definition of a multivariate spline \cite{whatspline}. For example, 
it is not obvious how to define a piecewise function in this context. 
The ``multivariate spline'' appears in \cite[Sec.~7.1.1, p.~113]{hyperplane}.  We refer to it here as ``multivariate truncated power'' to avoid confusion with other multivariate splines (for example, the box spline) and to be consistent with much of the literature \cite{Jia, diophantine, truncpowratpoly}. 
Elsewhere it is referred to as the ``cone spline'' \cite{Cones, compusimp}.
The multivariate truncated power is related to the \emph{multivariate B-spline}, which we do not define here.  
\end{remark}

For $s=2$, we can write
the multivariate truncated power in terms of the univariate Curry--Schoenberg B-spline. The following is the $s=2$ specialization of \cite[eq.~(2.10)]{compusimp}, which in turn is based on \cite[Thm.~1]{Cones}.
Keep in mind the convention \eqref{eq:SplineUnordered} for splines with unordered knots (Remark \ref{Remark:Unordered}).

\begin{theorem}\label{Theorem:BSplineConnection2}
Suppose $n_1,n_2,\ldots,n_k \in \R_{>0}$ and 
$A= 
\big[
\begin{smallmatrix}
m_1 & m_2 & \cdots & m_k \\
n_1 & n_2 & \cdots & n_k \\
\end{smallmatrix}
\big]
\in \M_{2\times k}(\R)$
has linearly independent rows.
Then
\begin{equation}\label{eq:TM}
T_A(\twovectorsmall{m}{n}) = \frac{(n)_{+}^{k-2}}{(k-1)!n_1n_2\cdots n_k} M\Big(\frac{m}{n}; \frac{m_1}{n_1}, \frac{m_2}{n_2}, \ldots, \frac{m_k}{n_k} \Big).  
\end{equation}
\end{theorem}

\subsection{Vector partition functions}\label{Subsection:PartitionFunction}
Let $A\in \M_{s\times k}(\Z)$ with columns $\vec a_1, \vec a_2, \dots, \vec a_k\in \Z^s$ such that $0\notin \conv\{\vec{a}_1, \vec{a}_2, \ldots, \vec{a}_k\}$. The \emph{vector partition function} 
\begin{equation}\label{eq:PartitionFunction} 
t_A(\vec b)= \big|\{\vec z \in \Z_{\geq 0}^k: A\vec z = \vec b\} \big|
\end{equation}
counts the number of ways $\vec{b}$ can be written as a nonnegative integer linear combination of $\vec{a}_1,\vec{a}_2,\ldots, \vec{a}_k$.
Vector partition functions appear in many fields of mathematics, including representation theory \cite{konstantLitRepTheory, goodRepLit}, symplectic geometry \cite{litonkonstant, KostantSymp}, differential geometry \cite{procesiLit}, commutative algebra \cite{StanleyCCA}, discrete geometry \cite{fancypartitions,matroids}, statistics \cite{MR1380519}, and more.  
In \cite{diophantine, Jia, truncpowratpoly, explicit}, the vector partition function is called the \emph{discrete truncated power} because its properties mirror those of the (continuous) multivariate truncated power function.
Another paper which looks at semigroup factorizations through vector partition functions is \cite{kfrobenius}.

\begin{remark}
If $A = [a_1~a_2~\ldots~a_k] \in \M_{1\times k}(\Z)$, then $t_A(n)$ is sometimes called the \emph{denumerant} of $\{a_1,a_2,\ldots,a_k\}$,  introduced by Sylvester in 1857 \cite{sylvester}; see also \cite{denum1, denum2, denum3}. 
\end{remark}

\subsection{Lattices}\label{Subsection:Lattices}
A \textit{lattice} is a discrete additive subgroup of $\R^n$ \cite{gruber_lek, Cassels}.
By \emph{discrete}, we mean that the points of the subgroup are separated by some fixed positive minimum distance. An \emph{integer lattice} is an additive subgroup of $\Z^n$ for some $n$.

A \emph{basis} for a lattice $\Lambda$ is a list of $\R$-linearly independent vectors in $\Lambda$ such that each element of $\Lambda$ is an integer linear combination of elements of the given list.
All bases of a lattice have the same cardinality, known as the \emph{rank} of the lattice. 
If the columns of $B \in \M_{s \times k}(\R)$ comprise a basis of a rank-$k$ lattice $\Lambda$, 
then $\Lambda=B\Z^k$.  The \emph{determinant} of $\Lambda$ is $\det \Lambda=\sqrt{\det(B^{\T} B)}$;
it is basis independent. 

Let $\E_{\Lambda}$ denote the vector subspace of $\R^s$ spanned by any basis of the rank-$k$ integer lattice $\Lambda \subset \Z^s$. 
This subspace is independent of the choice of basis; see \cite{ortholattices} or \cite[Lem.~4]{FukshanskyForst}.
Define $\overline{\Lambda} = \E_{\Lambda}\cap \Z^s$. 
If $\Lambda = \overline{\Lambda}$, then $\Lambda$ is \emph{primitive}  \cite{HeathBrownDiophantine}  (the term \emph{complete} is used in \cite{ortholattices}). Since each vector in $\E_{\Lambda}$ has a unique representation as an $\R$-linear combination of basis elements, a linear combination of basis elements which uses a non-integer coefficient is not in $\Lambda$. 
Corollary \ref{Corollary:EasyComplete} below furnishes examples of primitive lattices relevant to our investigations.

\begin{example}
The lattice $2\Z$ is not primitive.  The span of the basis $\{2\}$ is $\R$, so $\E_{2\Z} = \R$. Thus, $\overline{2\Z}=\E_{2\Z}\cap \Z =\R \cap \Z = \Z \neq 2\Z$, so $2\Z$ is not primitive.
\end{example}
 
\begin{example}\label{Example:Incomplete} 
Consider the integer lattice $\Lambda = [ \vec{m}\,\, \vec{n}] \Z^2 \subset \Z^4$ generated by 
$\vec{m} = \fourrowvector{1}{1}{1}{1}^{\T}$ and
$\vec{n} = \fourrowvector{1}{4}{7}{31}^{\T}$.
Since $\frac{1}{3}, \frac{2}{3} \notin \Z$ and $\frac{2}{3} \vec{m} + \frac{1}{2} \vec{n} = \fourrowvector{1}{2}{3}{11}^{\T}  \in \E_{\Lambda} \cap \Z^4 = \overline{\Lambda}$, it follows that $\Lambda$ is not primitive.
\end{example}
    
Let $\Lambda \subset \R^s$ be a lattice and let
$\E_{\Lambda}^{\perp}$ be the orthogonal complement of $\E_{\Lambda}$ with respect to the inner product on $\R^s$.  
The \emph{dual lattice} (or \emph{orthogonal lattice} \cite{ortholattices}) is 
\begin{equation*}
\Lambda^{\perp} = \E_{\Lambda}^{\perp}\cap \Z^s.
\end{equation*}
The next result is connected to the Brill-Gordan duality principle \cite{BrillGordan, FukshanskyForst, Gordan}
and can be found as \cite[Lem.~1, p.~2]{HeathBrownDiophantine}  or \cite[Thm.~1, p.~2]{ortholattices}. 

\begin{theorem}\label{Theorem:BrillGordan}
If $\Lambda$ is a primitive lattice, then $\det(\Lambda^{\perp}) = \det \Lambda$.
\end{theorem}

We say $A \in \M_{s \times k}(\Z)$ with $1\leq s<k$ is \emph{unimodular} if 
there is a $B \in \M_{(k-s) \times k}(\Z)$ such that $C = \twovectorsmall{A}{B} \in \M_k(\Z)$ has determinant $\pm 1$,
that is, the rows of $C$ comprise a basis for the integer lattice $\Z^k$. 
Let $\mathcal{J} = \{ I \subset \{1,2,\ldots,k\} : |I| = s \}$.
For each $I \in \mathcal{J}$, let $A_I$ be the $s \times s$ submatrix of $A$ whose columns are indexed by elements of $I$. Then \cite[Lem.~2, p.~15]{Cassels} tells us that $A$ is unimodular if and only if 
\begin{equation}\label{eq:GCD}
\gcd \{ \det A_I : I \in \mathcal{J} \} = 1.
\end{equation}

\begin{lemma}\label{Lemma:Complete}
Let $1\leq s < k$ and $A = \fourrowvector{ \vec{r}_1}{ \vec{r}_2}{\ldots}{ \vec{r}_s}^{\T}\in \M_{s\times k}(\Z)$.
 The following are equivalent.
\begin{enumerate}\addtolength{\itemsep}{3pt}
\item $A\Z^k = \Z^s$.
\item $A^{\T} \Z^s \subset \Z^k$ is a primitive lattice.
\item $A$ is unimodular.
\item $A$ satisfies \eqref{eq:GCD}.
\end{enumerate}
\end{lemma}

\begin{proof}
\noindent(a) $\Rightarrow$ (b) Let $\Lambda = A^{\T} \Z^s$.
Since $A^{\T}$ is an integer matrix, $\Lambda \subseteq \overline{\Lambda}$.
Let $\vec{x}\in \overline{\Lambda}\subseteq \Z^k$ and write $\vec{x} = \sum_{i=1}^s q_i \vec{r}_i$,
in which $q_1,q_2,\ldots, q_s\in \R$.  Since $A\Z^k=\Z^s$, for each $1\leq i \leq s$ there exists a $\vec{b}_i\in \Z^k$ such that $A\vec{b}_i = \vec{e}_i$, in which $\vec{e}_i \in \Z^s$ is the $i$th standard basis vector. Then 
$\vec{r}_j \cdot \vec{b}_i   =\delta_{ij}$ and hence each $q_i = \vec{x} \cdot \vec{b}_i \in \Z$.
Therefore, $\vec{x}\in A^{\T} \Z^s = \Lambda$, so $\Lambda = \overline{\Lambda}$. Thus, $\Lambda$ is primitive.  

\medskip\noindent(b) $\Rightarrow$ (c)
The $\vec{r}_1, \vec{r}_2, \ldots, \vec{r}_s$ form a basis for $\Lambda =A^{\T} \Z^s$. 
Since $\Lambda$ is primitive, this basis can be extended to a basis for $\Z^k$ \cite[Prop.~2.1]{Forst}, so $A$ is unimodular.

\medskip\noindent(c) $\Rightarrow$ (a)
If $A$ is unimodular, there is a $B \in \M_{ (k-s) \times k}(\Z)$ so that $\twovectorsmall{A}{B} \in \M_k(\Z)$
has determinant $\pm 1$. Then
$\Z^k 
= \big\{ \twovectorsmall{A}{B} \vec{z} : \vec{z} \in \Z^k \big\} 
= \big\{ \twovectorsmall{A\vec{z}}{B \vec{z}} : \vec{z} \in \Z^k \big\}$,
so $A\Z^k = \Z^s$.

\medskip\noindent(c) $\Leftrightarrow$ (d)
This is well known \cite[Lem.~2, p.~15]{Cassels}.
\end{proof}

\begin{corollary}\label{Corollary:EasyComplete}
Let $\Lambda = \vec{n}\Z$ with $\vec{n}=[n_i] \in \Z^k$.
Then $\Lambda$ is primitive if and only if $\gcd(n_1, n_2, \dots, n_k)=1$.
\end{corollary}

\begin{corollary}\label{Corollary:Delta}
Let $A = \tworowvector{\vec{m}}{\vec{n}}^{\T} \in \M_{2 \times k}(\Z)$, in which $\vec{m}^{\T} = \fourrowvector{1}{1}{\ldots}{1}$.
Then
\begin{equation*}
A \Z^k =\Z^2 
\quad\iff\quad
\gcd(n_2-n_1,\, n_3-n_2,\ldots,\, n_k-n_{k-1}) = 1. 
\end{equation*}
\end{corollary}

\begin{proof}
Lemma \ref{Lemma:Complete} says $A\Z^k=\Z^2$ if and only if $A$ satisfies \eqref{eq:GCD}. 
The $2\times 2$ minor of $A$ using columns $i$ and $j$ is $n_j-n_i$, so \eqref{eq:GCD} becomes 
$\gcd\{ n_j-n_i:i\neq j\}=1$. 
Since
$ n_j-n_i=(n_j-n_{j-1})+(n_{j-1}-n_{j-2})+\cdots +(n_{i+1}-n_i)$
it follows that \eqref{eq:GCD} is equivalent to $\gcd(n_2-n_1, n_3-n_2,\ldots, n_k-n_{k-1}) = 1$.
\end{proof}

\begin{remark}
Numerical semigroup theory uses the convention  $1 \leq n_1 < n_2< \cdots < n_k$.  In that context, 
$\delta = \gcd(n_2-n_1, n_3-n_2,\ldots, n_k-n_{k-1})$
is the minimum element of the \textit{delta set} of $\langle n_1, n_2,\ldots, n_k\rangle$ \cite{delta,delta1, delta2, delta3, delta4, delta5}.  
\end{remark}

One might wish $n_1,n_2,\ldots,n_k$ to be unordered.  In this case, the next lemma ensures that 
greatest common divisor in Corollary \ref{Corollary:Delta} is permutation invariant.

\begin{lemma}
Let $n_1,n_2\dots, n_k\in \Z$.  Then 
$\gcd(n_2-n_1,n_3-n_2,\ldots, n_k-n_{k-1})=\gcd(n_{\pi(2)}-n_{\pi(1)},n_{\pi(3)}-n_{\pi(2)},\ldots, n_{\pi(k)}-n_{\pi(k-1)})$
for any permutation $\pi$ of $\{1,2,\ldots,n\}$.
\end{lemma} 

\begin{proof}
It suffices to prove that $d=\gcd(n_2-n_1,n_3-n_2,\ldots, n_k-n_{k-1})$ divides $d' = \gcd(n_{\pi(2)}-n_{\pi(1)},n_{\pi(3)}-n_{\pi(2)},\ldots, n_{\pi(k)}-n_{\pi(k-1)})$ since then symmetry will ensure $d' \mid d$.  Observe that
\begin{equation*}
n_{\pi(i)}-n_{\pi(i-1)} = \pm \sum_{j=\pi(i-1)+1}^{\pi(i)} ( n_j-n_{j-1} ),
\end{equation*}
with the plus sign if $\pi(i) > \pi(i-1)$ and the negative sign if $\pi(i) < \pi(i-1)$.  Thus,
$d \mid (n_{\pi(i)}-n_{\pi(i-1)})$ for $i=1,2,\ldots,k$, so $d \mid d'$.
\end{proof}

\subsection{From volume to point counts}\label{Subsection:Volume}
Theorem \ref{Theorem:BodyCount} below connects the $s$-dimensional volume of a subset of an $s$-dimensional affine subspace of $\R^k$ to the number of integer lattice points inside it.  The next definition is \cite[Def.~2.2]{widmer}.

\begin{definition}
    A set $\S \subseteq \R^k$ is of \emph{Lipschitz class} $\Lip(k,M,L)$ if there are $N$ maps $\phi_1,\phi_2,\ldots, \phi_N: [0,1]^{k-1}\to \R^{k}$ satisfying the Lipschitz condition
\begin{equation*}
    |\phi_i(\vec{x})-\phi_i(\vec{y})| \leq L |\vec{x} -\vec{y}|
    \quad \text{for $\vec{x},\vec{y}\in [0,1]^{k-1}$ and $i=1,2,\ldots, M$},   
\end{equation*}
such that $\S$ is covered by the images of the $\phi_i$.
\end{definition}

The following is essentially \cite[Thm.~2.6]{widmer}.

\begin{theorem}\label{Theorem:WidmerLip}
If $\S \subset \R^k$ is convex and in a ball of radius $r$, then $\partial \S \in \Lip(k,1,8k^{5/2}r)$.
\end{theorem}

Here $\partial \S$ denotes the boundary of a set $\S$.
In \cite[Thm.~6]{widmer}, the convexity hypothesis is replaced with the condition that $\S$ be of narrow class $1$ (a subset of $\R^n$ is of \textit{narrow class} $\ell$ if it is bounded, measurable, and intersects every line in at most $\ell$ intervals or single points, and if the same is true of any projection of that subset onto any linear subspace of $\R^n$).  As noted after the definition of narrow class \cite[Def.~2.1]{widmer}, sets of narrow class $1$ are exactly bounded convex sets. 

Let $\h$ be an $s$-dimensional subspace of $\R^k$. 
The \emph{successive minima} of the lattice $\Lambda =\h\cap \Z^k$ 
are $\lambda_1 \leq \lambda_2 \leq \cdots \leq \lambda_k$, defined by
\begin{equation*}
\lambda_i = \inf \left\{ r \in \R_{>0} : \dim_{\R} \spn_{\Z} (B_{\h}(r) \cap \Lambda) \geq i \right\};
\end{equation*}
here $B_{\h}(r)$ is a ball of radius $r$ centered at the origin in $\h$ \cite{gruber_lek}.
The infimum above is achieved since $\Lambda$ is discrete. 
In other words, $\lambda_i$ measures 
how much a unit ball centered at the origin must be expanded to contain at least $i$ linearly independent elements of $\Lambda$.  
Since $\Lambda \subseteq \Z^k$, we see that $\lambda_i \geq 1$ for all $1 \leq i \leq k$. 

The following is \cite[Thm.~2.4]{widmer} and \cite[Thm.~5.4]{widmer2}.
If $i=0$, the expression in the maximum below is interpreted as $1$.

\begin{theorem}\label{Theorem:WidmerCount}
    Let $\Lambda \subset \R^k$ be a lattice with successive minima $\lambda_1,\lambda_2,\ldots, \lambda_k$. 
    Let $\S \subset \R^k$ be bounded with $\partial\S \in \Lip(k,M,L)$. Then $\S$ is measurable and 
    \begin{equation*}
    \left| | \S\cap \Lambda| -\frac{\vol_k \S}{\det \Lambda}\right|
    \leq M k^{3k^2/2}  \max_{0\leq i<k} \frac{L^i}{\lambda_1 \lambda_2  \cdots \lambda_i}. 
    \end{equation*}
\end{theorem}

The \emph{diameter} of a bounded subset $\S$ of a Euclidean space is
\begin{equation*}
\diam \S = \sup \{ \|\vec{x}-\vec{y}\| : \vec{x},\vec{y} \in \S \}.
\end{equation*}
Theorems \ref{Theorem:WidmerLip} and \ref{Theorem:WidmerCount}, which are \cite[Thms.~2.4 \& 2.6]{widmer}, provide the following.

\begin{theorem}\label{Theorem:BodyCount}
Let $\h \subseteq \R^k$ be an $s$-dimensional subspace with $s \geq 1$ and
let $\Lambda = \h \cap \Z^k$ be a rank-$s$ lattice in $\h$. 
Let $\S \subset \h$ be a convex set with diameter $d$ and nonzero volume $\vol_s \S$ in $\h$. Then
\begin{equation*}
\left| |\S\cap \Z^k|- \frac{\vol_s \S}{\det \Lambda}  \right| \leq (4 \sqrt{2} d)^{s-1} s^{\frac{3s^2+5s-5}{2}}.
\end{equation*}
\end{theorem}

\begin{proof}
Theorem \ref{Theorem:WidmerCount} concerns a full-rank lattice in $\R^k$ and Theorem \ref{Theorem:WidmerLip} a convex set in $\R^k$.
Since $\h$ is an $s$-dimensional Euclidean space, endowed with $s$-dimensional Lebesgue measure, and $\Lambda$ is a full-rank lattice in $\h$, the results carry over. 
Since $\S$ is convex with finite diameter $d$, it is compact.  Jung's theorem ensures that
$\S$ lies in a closed ball of radius $d( \frac{s}{2(s+1)})^{1/2} \leq d/\sqrt{2}$ \cite{Jung1,Jung2}.
Theorem \ref{Theorem:WidmerLip} yields $\partial\S \in \Lip(s,1,4\sqrt{2}s^{5/2}d)$.
Since $\Lambda=\h\cap \Z^k$ and $\S\subset \h$, we have $\S \cap \Lambda = \S \cap \Z^k$. 
The successive minima for $\Lambda$ are at least $1$, so Theorem \ref{Theorem:WidmerCount} implies
\begin{equation*}
    \left| |\S\cap \Z^k|- \frac{\vol_s \S}{\det \Lambda}  \right|
    \leq  s^{3s^2/2} (4 \sqrt{2} s^{5/2}d)^{s-1}
    =(4 \sqrt{2} d)^{s-1}  s^{\frac{3s^2 +5s-5}{2}}. \qedhere
\end{equation*}
\end{proof}

To apply Theorem \ref{Theorem:BodyCount} later on in the proofs of Theorems \ref{Theorem:A} and \ref{Theorem:B}, we need to bound the diameter of the set $\ZZ_S(n) \subset \Pi_{\vec{n}^{\T}}$ of factorizations of $n$, in which $S = \langle n_1,n_2,\ldots,n_k \rangle$ and $\vec{n} = [n_i] \in \Z^k$ with $1 \leq n_1 \leq n_2 \leq \cdots \leq n_k$.

\begin{lemma}\label{Lemma:InBall}
Let $\vec{n} = [n_i] \in \Z_{>0}^k$ and $n \in \Z_{\geq 0}$.
Then
\begin{equation*}
\Pi_{\vec{n}^{\T}}(n)=\{\vec{x}\in \R_{\geq 0}^k:  \vec{n} \cdot \vec{x} = n\}
\end{equation*}
is convex and contained in a ball with diameter $\sqrt{2}n$.
\end{lemma}

\begin{proof}
We claim that $\Pi_{\vec{n} ^{\T}}=\conv \mathfrak{B}$, in which
$\mathfrak{B} = \{\frac{n}{n_1}\vec{e}_1,\frac{n}{n_2}\vec{e}_2, \ldots, \frac{n}{n_k}\vec{e}_k\}$.
First, $\Pi_{\vec{n}^{\T}}(n)$ is convex since it is the intersection of $k$ half spaces (coming from the nonnegativity of each coordinate of $\vec{x}$) in the affine $(k-1)$-dimensional hyperplane defined by $\vec{n} \cdot \vec{x} = n$. 
Since $ \vec{n} \cdot \frac{n}{n_i}\vec{e}_i  = n$ for each $i=1,2,\ldots, k$, we see that $\frac{n}{n_i}\vec{e}_i\in \Pi_{\vec{n}^{\T}}(n)$, so $\conv \mathfrak{B}\subseteq \Pi_{\vec{n}^{\T}}(n)$. If
$\vec{x}= \fourrowvector{x_1}{x_2}{\ldots}{x_k}^{\T} \in \Pi_{\vec{n}^{\T}}(n)$, then 
\begin{equation*}
\vec{x}=\sum_{i=1}^k t_i \Big(\frac{n}{n_i}\vec{e}_i \Big),
\quad \text{in which $t_i=\frac{x_in_i}{n} \geq 0$ for $i=1,2,\ldots,k$}.
\end{equation*}
Since
\begin{align*}
    \sum_{i=1}^k t_i &= \sum_{i=1}^k \frac{x_in_i}{n}
    =\frac{1}{n}\sum_{i=1}^k n_i x_i
    =\frac{1}{n} \vec{n} \cdot \vec{x}
    = \frac{1}{n}n
    = 1,
\end{align*}
$\vec{x} \in \conv \mathfrak{B}$. Thus, $\Pi_{\vec{n}^{\T}}(n)\subseteq \conv \mathfrak{B}$ and hence $\Pi_{\vec{n}^{\T}}(n) = \conv \mathfrak{B}$.
The diameter of a simplex is the maximum distance between any two vertices \cite[Sec.~2.1]{HatcherBook}, so 
\begin{equation}\label{eq:CouldDo}
\diam(\Pi_{\vec{n}^{\T}}(n)) \leq \max_{i\neq j} \bigg(\frac{n^2}{n_i^2}+\frac{n^2}{n_j^2} \bigg)^{1/2} 
\leq \sqrt{2}n
\end{equation}
as claimed.
\end{proof}

\section{Theorem \ref{Theorem:A}: weighted factorization length statistics and B-splines}\label{Section:ProofSemigroup}

Given a numerical semigroup $S = \langle n_1,n_2,\ldots,n_k\rangle$, in which the generators 
$1 \leq n_1 \leq n_2 \leq \cdots \leq n_k$ may be repeated or non-minimal, Theorem \ref{Theorem:A} 
bounds the difference between an integral of a corresponding B-spline
and the number of factorizations $\vec{x}$ of $n$
such that an associated weighted factorization length  $\vec{x} \mapsto \vec{m} \cdot \vec{x}$ 
has values in $[\alpha n , \beta n]$.  Our bound is explicit: there are no unspecified implied constants.  
As $n \to \infty$, we obtain asymptotics for arbitrary factorization-function statistics, a broad generalization
of the main result of \cite{semigroupsII}.

\begin{thmx}\label{Theorem:A}
Let $S = \langle n_1,n_2,\ldots,n_k \rangle \subseteq \Z_{\geq 0}$ with 
\begin{equation*}
n_1 \leq n_2\leq \cdots \leq n_k
\quad \text{and} \quad \gcd(n_1, n_2, \dots,n_k)=1. 
\end{equation*}
If $\vec{m}= [m_i]  \in \R^k$ and $\vec{n} = [n_i] \in \Z_{>0}^k$ are linearly independent, then
for all $n\in \Z_{>0}$ and $\alpha,\beta\in \R \cup \{\pm \infty\}$ with $\alpha \leq  \beta$,
{\small
\begin{equation}\label{eq:Concrete}
\left|    \Big|\!\{\vec x\in \ZZ_S(n): \vec{m} \cdot (\vec{x} / n ) \in [\alpha , \beta ]\} \!\Big|  
- 
n^{k-1}\!\! \int_{\alpha}^{\beta} \frac{M\big(t; \frac{m_1}{n_1}, \frac{m_2}{n_2}, \ldots, \frac{m_k}{n_k} \big)}{(k-1)!n_1n_2\cdots n_k}\,dt \right| \leq n^{k-2} E_1(k),
\end{equation}
}%
in which $M\big(t; \frac{m_1}{n_1}, \frac{m_2}{n_2}, \ldots, \frac{m_k}{n_k} \big)$ is a Curry--Schoenberg B-spline and
\begin{equation}\label{eq:E1}
E_1(k)=  8^{k-2} (k-1)^{\frac{3k^2-k-7}{2}}.
\end{equation}
If $f:\R\to\R$ is bounded and continuous, then
\begin{equation}\label{eq:FancyStats}
\lim_{n\to\infty} \frac{1}{|\ZZ_S(n)|} \sum_{ \vec{x} \in \ZZ_S(n)} f \Big( \frac{\vec{m} \cdot \vec{x} }{n} \Big) =
 \int_{-\infty}^{\infty} \, f(t) M\Big(t; \frac{m_1}{n_1}, \frac{m_2}{n_2}, \ldots, \frac{m_k}{n_k} \Big)\,dt.
\end{equation}
\end{thmx}

Since Section \ref{Section:Examples} contains many detailed examples of Theorem \ref{Theorem:A} in action, we defer
the exploration of applications until then.
The present section is devoted to the proof of Theorem \ref{Theorem:A}, although we first make some preliminary  remarks.

\begin{remark}
The B-spline in Theorem \ref{Theorem:A} is well defined since
the linear-independence hypothesis ensures that 
$\frac{m_1}{n_1}, \frac{m_2}{n_2},\ldots, \frac{m_k}{n_k}$ are not all equal.
Also recall that \eqref{eq:SplineUnordered} permits knots to be unordered.
Otherwise, we would need 
\begin{equation*}
\frac{m_1}{n_1} \leq  \frac{m_2}{n_2} \leq \cdots \leq \frac{m_k}{n_k},
\end{equation*}
which might entail reordering $n_1 \leq n_2 \leq \cdots \leq n_k$, which violates the standard notation used in numerical semigroup theory.
Finally, note that if $n_1= n_2 = \cdots = n_k$, then the condition $\gcd(n_1,n_2,\ldots,n_k)=1$ forces $n_1=n_2 =\cdots = n_k = 1$.  
Nevertheless, this case is still of interest; see Example \ref{Example:Picnic}.
\end{remark}

\begin{remark}
An important step in the proof of Theorem \ref{Theorem:A} is an explicit estimate \eqref{eq:ExplicitZ} 
for the number $|\ZZ_S(n)|$ of factorizations of $n$ (the Sylvester \emph{denumerant}):
\begin{equation*}
\left\lvert |\ZZ_S(n)| - 
\frac{n^{k-1}}{(k-1)!n_1n_2\cdots n_k}
  \right \rvert \leq n^{k-2}E_1(k).
\end{equation*}
The asymptotic version of this result is due to Schur \cite{schur1926additiven}; see also
\cite[Thm.~4.2.1]{AlfonsinDiophantineFrobenius}, \cite{comtet,semigroupsII}.
The explicit error bound, which depends only upon the number of generators, appears novel,
although relevant Big-O estimates exist \cite{NathansonParts}.
\end{remark}

\begin{remark}
Working through the proof below and tracing everything back to \eqref{eq:CouldDo} reveals that the constant 
$8$ in \eqref{eq:E1} can be replaced by $4\sqrt{2} (n_1^{-2} + n_2^{-2})^{1/2}$.  Since this introduces some dependence on 
the smallest generators $n_1$ and $n_2$, we use \eqref{eq:E1} instead, which depends only upon $k$. Determining the actual circumradius instead of bounding it with Jung's theorem would give even better bounds.
\end{remark}

The remainder of this section is devoted to the proof of Theorem \ref{Theorem:A},
which we divide into several steps.  Subsection \ref{Subsection:ExDisVolEst} contains an explicit estimate (Lemma \ref{Lemma:AlmostDone}) 
that relates the volume of a certain polytope to its discrete volume.  In Subsection \ref{Subsection:Involve} we evaluate
the volume of our polytope in terms of a multivariate truncated power function (Lemma \ref{Lemma:Cavalieri}).
We wrap everything up in Subsection \ref{Subsection:FinalConclude}.

\subsection{An explicit discrete-volume estimate}\label{Subsection:ExDisVolEst}
Our first lemma reduces the proof of Theorem \ref{Theorem:A} to a volume computation for a certain polytope.  Observe that the cardinality of 
the set $\QQ_n(\alpha,\beta) \cap \Z^k$ in the lemma, that is, the discrete volume of the polytope $\QQ_n(\alpha,\beta)$, is the quantity
$|\{\vec x\in \ZZ_S(n): \vec{m} \cdot \frac{\vec{x}}{n} \in [\alpha , \beta ]\} |  $ from the left side of \eqref{eq:Concrete}
in the statement of Theorem \ref{Theorem:A}.

\begin{lemma}\label{Lemma:AlmostDone}
Let $\vec{m}=[m_i]\in \R^k$ and
$\vec{n} =[n_i] \in \Z_{>0}^k$ be linearly independent with $\gcd(n_1, n_2,\dots,n_k)=1$.
For $n\in \Z_{>0}$ and $\alpha,\beta\in \R\cup \{\pm \infty\}$,
define the affine subspace 
\begin{equation*}
\K=\{\vec{x}\in \R^k:  \vec{n} \cdot  \vec{x}  = n\}
\end{equation*}
and the polytope
\begin{equation}\label{eq:Qn}
\QQ_n(\alpha,\beta) = \big\{\vec{x}\in \R_{\geq 0}^k:  
\text{$\vec{n} \cdot  \vec{x} = n$ and $\vec{m} \cdot \vec{x} \in[\alpha n, \beta n]$} \big\} 
\subset \K.
\end{equation}
Then
\begin{equation*}
\left| |\QQ_n(\alpha,\beta) \cap \Z^k| - \frac{\vol_{k-1}\QQ_n(\alpha,\beta)}{\norm{\vec{n}}} \right| \leq n^{k-2}E_1(k),
\end{equation*}
in which $E_1(k)$ is given by \eqref{eq:E1}.
\end{lemma}

\begin{proof}
We may assume $\alpha \leq \beta$.  First observe that $\QQ_n(\alpha,\beta) \subseteq \QQ_n(\alpha',\beta')$ whenever $[\alpha,\beta] \subseteq [\alpha', \beta']$.
In particular, each $\QQ_n(\alpha,\beta) \subseteq \QQ_n(-\infty,\infty)$.

If $\QQ_n(\alpha,\beta) = \varnothing$, the desired inequality holds, so we may assume $\QQ_n(\alpha,\beta) \neq \varnothing$.
The definition of $\QQ_n(\alpha,\beta)$ ensures that $\frac{\vec{m}}{n}\cdot \vec{x} \in [\alpha,\beta]$
for all $\vec{x} \in \QQ_n(\alpha,\beta)$.
Since the functional $\vec{x}\mapsto \frac{\vec{m}}{n}\cdot \vec{x}$ is linear, and hence convex, it achieves its maximum and minimum values on the polytope $\QQ_n(-\infty,\infty)$ at vertices, and these are of the form $\frac{n}{n_i}\vec{e}_i$ where $\vec{e}_i$ is a standard basis vector. 
Since $\frac{\vec{m}}{n} \cdot\frac{n}{n_i}\vec{e}_i=\frac{m_i}{n_i}$, we have
\begin{equation*}
\QQ_n(\alpha, \beta)=\QQ_n(\max(\alpha,\gamma_{\min}),\min(\beta, \gamma_{\max})),
\end{equation*}
in which
\begin{equation*}
\gamma_{\min} = \min_{1\leq i\leq k} \frac{m_i}{n_i}
\quad\text{and}\quad
\gamma_{\max} = \max_{1\leq i\leq k} \frac{m_i}{n_i}.
\end{equation*}
Hence we may assume $[\alpha,\beta] \subseteq [\gamma_{\min}, \gamma_{\max}]$.

We additionally assume that $\alpha<\beta$, which ensures $\vol_{k-1}\QQ_n(\alpha,\beta)>0$.
We return to the $\alpha=\beta$ case later in the proof. 

Since $\gcd(n_1,n_2,\dots,n_k)=1$, B\'ezout's identity provides some $\vec{p}\in \Z^k$ such that $ \vec{n} \cdot \vec{p}  = n$. 
Define $\phi(\vec{x})= \vec{x} - \vec{p}$ and
\begin{equation*}
\K_0 = \{\vec{x}\in \R^k:  \vec{n} \cdot \vec{x} = 0\},
\end{equation*}
so $\phi(\K) = \K_0$.
Then $\phi(\QQ_n(\alpha,\beta)) \subset \K_0$ has the same volume, diameter, and number of integer points as $\QQ_n(\alpha,\beta)$.
Observe that the vectors $n_{i+1}\vec{e}_{i}-n_{i}\vec{e}_{i+1}$ for $i=1,\dots,n-1$ give an integer basis for the $(k-1)$-dimensional subspace $\K_0\subset \R^k$ and hence $\K_0\cap \Z^k$ is a rank $k-1$ lattice within $\K_0$. Moreover,  $\phi(\QQ_n(\alpha,\beta))$ is convex subset of $\K_0$.

Corollary \ref{Corollary:EasyComplete}  ensures that $\vec{n}\Z$ is a primitive lattice, 
so Theorem \ref{Theorem:BrillGordan} implies
\begin{equation*}
\det(\K_0\cap \Z^k)=\det((\vec{n}\Z)^{\perp})=\det(\vec{n}\Z)=\sqrt{\vec{n}^{\T}\vec{n} }= \norm{\vec{n}}.
\end{equation*}
Lemma \ref{Lemma:InBall} ensures that
$\diam\phi(\QQ_n(\alpha,\beta)) \leq \sqrt{2}n$, so Theorem \ref{Theorem:BodyCount} yields
\begin{align*}
&\left| | \QQ_n(\alpha,\beta) \cap \Z^k| -\frac{\vol_{k-1}\QQ_n(\alpha,\beta)}{\norm{\vec{n}}}\right|
=\left| |\phi(\QQ_n(\alpha,\beta))\cap \Z^k| - \frac{\vol_{k-1}\phi(\QQ_n(\alpha,\beta))}{\det(\h\cap \Z^k)}\right| \\
&\qquad\qquad\qquad\leq (4 \sqrt{2}\diam\QQ_n(\alpha,\beta))^{(k-1)-1}(k-1)^{\frac{3(k-1)^2+5(k-1)-5}{2}}\\
&\qquad\qquad\qquad=(8n)^{k-2}(k-1)^{\frac{3k^2-k-7}{2}}\\
&\qquad\qquad\qquad=n^{k-2}E_1(k) .
\end{align*}

This proves the lemma whenever $\gamma_{\min}\leq \alpha<\beta\leq \gamma_{\max}$. It remains to handle the case where $\alpha=\beta$. 
If $|\QQ_n(\beta,\beta) \cap \Z^k| = 0$, then the lemma is trivial in this case so we suppose that $|\QQ_n(\beta,\beta) \cap \Z^k| \neq 0$. 

We show that for small $\delta$, $\QQ_n(\beta-\delta, \beta)$ has the same integer points as $\QQ_n(\beta,\beta)$.
Since the argument above proves the lemma for $\QQ_n(\alpha,\beta)$ where $\alpha\neq \beta$, we can thus apply it to $\QQ_n(\beta-\delta,\beta)$ to handle the case of $\QQ_n(\beta,\beta)$. The set $\K\cap \Z_{\geq 0}^k$ is finite due to the positivity of the entries of $\vec{n}$. Let $a$ be the largest element of the finite set $\{ \vec{m}\cdot \vec{z}: \vec{z}\in \K\cap \Z_{\geq 0}^k\}$ which is less than $\beta$ or $\beta-1$ if nothing in the set is less than $\beta$. Then $\QQ(\beta-\delta,\beta)\cap \Z^k = \QQ(\beta, \beta)\cap \Z^k$ for any $\delta>0$ such that $\beta-\delta>a$. 

For any $\epsilon>0$, there exists a $\delta>0$ such that 
\begin{equation*}
\frac{ \vol_{k-1}\QQ_n( \beta -\delta,\beta)  }{ \norm{\vec{n}} }
\leq  \min\big\{ \epsilon , \, 1 \, \beta-\gamma_{\min} \big\}.
\end{equation*}
Note that this implies $\frac{ \vol_{k-1}\QQ_n( \beta -\delta,\beta)  }{ \norm{\vec{n}} }<1 \leq | \QQ_n (\beta -\delta,\beta )\cap \Z^k |$.

We have already shown that the lemma holds whenever $\gamma_{\min}\leq \alpha < \beta\leq \gamma_{\max}$. 
If $\beta> \gamma_{\min}$, then 
\begin{align*}
&\left| \big| \QQ_n( \beta ,\beta)\cap \Z^k \big|-\epsilon\right| \\
&\qquad \leq \left| \big| \QQ_n ( \beta -\delta,\beta)\cap \Z^k\big|
- \frac{\vol_{k-1}\QQ_n(\beta - \delta,\beta) }{\norm{\vec{n}} }\right|\\
&\qquad \leq n^{k-2}E_1(\vec n),
\end{align*} 
so letting $\epsilon \to 0^+$ recovers the desired result in the $\alpha =  \beta>\gamma_{\min}$ case. A similar argument handles the case $\alpha=\beta = \gamma_{\min}$. Thus, the lemma holds for all $\alpha,\beta$.
\end{proof}

\subsection{Involving the multivariate truncated power}\label{Subsection:Involve}
The next result helps compute $\vol_{k-1} \QQ_n(\alpha,\beta)$ in applications of Lemma \ref{Lemma:AlmostDone}.
We do this by using a Cavalieri-type principle to relate its value to a multivariate truncated power (Subsection \ref{Subsection:Truncated}), which in the conclusion of the proof (Subsection \ref{Subsection:FinalConclude}) we evaluate in terms of a Curry--Schoenberg B-spline.

\begin{lemma}\label{Lemma:Cavalieri}
Suppose $A= \tworowvector{ \vec{m} }{ \vec{n}}^{\T} \in \M_{2\times k}(\R)$ has linearly independent rows and 
$\vec{0}$ is not in the convex hull of its columns.
Let $\alpha\leq\beta$ be in $\R \cup \{ \pm \infty\}$ 
and let
\begin{equation*}
\QQ_1(\alpha, \beta) = \big\{\vec{x}\in \R_{\geq 0}^k:  
\text{$\vec{n} \cdot  \vec{x} = 1$ and $\vec{m} \cdot \vec{x} \in[\alpha, \beta]$} \big\} 
\end{equation*}
as in \eqref{eq:Qn}.
Then
\begin{equation*}
\vol_{k-1}\QQ_1(\alpha, \beta) = \norm{ \vec{n} }\int_{\alpha}^{\beta} T_A\big(\twovectorsmall{t}{1}\big) \,dt,
\end{equation*}
in which $T_A$ denotes the corresponding multivariate truncated power.
\end{lemma}

\begin{proof}
Let $P \in \M_k(\R)$ denote the orthogonal projection of $\R^k$ onto $\{\vec{n}\}^{\perp}$:
\begin{equation*}
P\vec{x} =\vec{x} - \frac{ \vec{x} \cdot  \vec{n}}{\norm{\vec{n}}^2}\vec{n}.
\end{equation*}
Extend
\begin{equation*}
\vec{b}_1=\frac{\vec{n}}{\norm{\vec{n}}} 
\quad \text{and}\quad 
\vec{b}_2=\frac{P\vec{m}}{\norm{P\vec{m}}} 
=  \frac{ \vec{m}-(\vec{m} \cdot \vec{b}_1)\vec{b}_1 }{ \norm{P\vec{m}} }
\end{equation*}
to an orthonormal basis $\vec{b}_1, \vec{b}_2,\ldots, \vec{b}_k$ for $\R^k$ and let $B$ be the real orthogonal matrix
$\fourrowvector{ \vec{b}_1 }{ \vec{b}_2 }{ \ldots }{ \vec{b}_k}  \in \M_{k\times k}(\R)$.
Then $\vec{b}_i \cdot \vec{m} = 0$ for $3 \leq i \leq k$ since $\vec{m}$ is a linear combination of $\vec{b}_1$ and $\vec{b}_2$.
Write $\vec{x} = B \vec{z}$, in which 
$\vec{z} = \fourrowvector{z_1}{z_2}{\ldots}{z_k}^{\T}$, so the condition $\vec{n} \cdot \vec{x} =1$ becomes
\begin{equation*}
1= \vec{n} \cdot \vec{x} = \norm{ \vec{n} } (\vec{b}_1 \cdot \vec{x})
= \norm{ \vec{n} } (\vec{b}_1 \cdot B \vec{z})
= \norm{ \vec{n} } z_1.
\end{equation*}
Moreover, 
\begin{equation*}
\vec{m} \cdot \vec{x}
= \vec{m} \cdot B\vec{z} 
= \vec{m} \cdot (z_1 \vec{b}_1 + z_2 \vec{b}_2) 
= \vec{m} \cdot \bfell( z_2 ),
\end{equation*}
in which
\begin{equation*}
\bfell(s) =\frac{\vec{n}}{\norm{\vec{n}}^2}+s\frac{P\vec{m}}{\norm{P\vec{m}}}.
\end{equation*}

Let $\chi_{X}(\cdot)$ denote the characteristic function of a set $X$.  Then
\begin{align}
&\vol_{k-1}\QQ_1(\alpha, \beta)
= \int_{\vec{x} \,:\,  \vec{n}  \cdot \vec{x} = 1}\chi_{\R_{\geq 0}^k}(\vec{x})\chi_{[\alpha, \beta]}( \vec{m} \cdot \vec{x} )\,d\vec{x} \nonumber\\
&\quad=\int_{\vec{z} \,:\,  \vec{n}  \cdot (B\vec{z}) =1} \chi_{\R_{\geq 0}^k}(B\vec{z}) \chi_{[\alpha, \beta]}( \vec{m}  \cdot B\vec{z} )\,d \vec{z} \nonumber\\
&\quad=\int_{ \R^{k-1}} \chi_{\R_{\geq 0}^k}\bigg(\frac{\vec{n}}{\norm{\vec{n}}^2}+\sum_{i=2}^k z_i\vec{b}_i\bigg)\chi_{[\alpha, \beta]}( \vec{m} \cdot \bfell(z_2) )\,d z_3 \,d z_4 \cdots \,d z_k \,d z_2 \nonumber\\
&\quad=\int_{\R}\int_{\R^{k-2}} \chi_{\R_{\geq 0}^k}\bigg(\frac{\vec{n}}{\norm{\vec{n}}^2}+\sum_{i=2}^k z_i\vec{b}_i\bigg) 
\chi_{[\alpha,\beta]}\big( \vec{m} \cdot \bfell(z_2) \big) \,d z_3 \,d z_4 \cdots \,d z_k \,d z_2 \nonumber\\
&\quad=\!\!\int_\R \chi_{[\alpha,\beta]}\big( \vec{m} \cdot \bfell(z_2)\big) \!\!\int_{\R^{k-2}} \chi_{\R_{\geq 0}^k}\bigg(\frac{\vec{n}}{\norm{\vec{n}}^2} 
+ \sum_{i=2}^k z_i\vec{b}_i\bigg) \,d z_3 \,d z_4 \cdots \,d z_k \,d z_2. \label{eq:Fubinied}
\end{align}
The second equality above holds because $B$ is a real-orthogonal matrix and hence the Jacobian for the change of variables is $1$.

Since $B^{-1}$ is a real orthogonal matrix, the $(k-2)$-dimensional induced volume of 
$\QQ_1(\alpha, \beta) \cap \{\vec{x}: \vec{m}\cdot \vec{x} = \vec{m} \cdot \bfell(s)\}$ equals that of 
\begin{equation*}
B^{-1}(\QQ_1(\alpha, \beta) \cap \{\vec{x}: \vec{m}\cdot \vec{x} = \vec{m} \cdot \bfell(s)\}).
\end{equation*}
Since $\vec{0}$ is not in the convex hull of the columns of $A$, the variable polytope
\begin{equation*}
\QQ_1(\alpha, \beta) \cap \{\vec{x} \in \R_{\geq 0}^k: \vec{m}\cdot \vec{x} = \vec{m} \cdot \bfell(s)\} = \Pi_A \big( \twovectorsmall{\vec{m} \cdot \bfell(s)}{1} \big)
\end{equation*}
is bounded and hence has finite volume.  For $\alpha\leq s \leq \beta$, 
\begin{equation*}
B^{-1} (\QQ_1(\alpha, \beta) \cap \{\vec{x}: \vec{m}\cdot \vec{x} = \vec{m} \cdot \bfell(s)\})
\end{equation*}
is the set of points $\vec{z}=(z_1,z_2,\ldots z_k)$ such that $z_1=\norm{\vec{n}}^{-1}$, $z_2=s$, and $B\vec{z} \in \R^{k}_{\geq 0}$, in which case
$B\vec{z} = \frac{\vec{n}}{\norm{\vec{n}}^2}+\sum_{i=2}^{k}z_i\vec b_i$. Thus, the $(k-2)$-dimensional induced 
volume of  $\QQ_1(\alpha, \beta) \cap \{\vec{x}: \vec{m}\cdot \vec{x} = \vec{m} \cdot \bfell(z_2)\}$ is
\begin{equation*}
\chi_{[\alpha,\beta]}\big( \vec{m} \cdot \bfell(z_2)\big) 
\int_{\R^{k-2}} \chi_{\R_{\geq 0}^k}\left(\frac{\vec{n}}{\norm{\vec{n}}^2}+ \sum_{i=2}^k z_i \vec{b}_i\right) \,d z_3 \,d z_4 \dots \,d z_k .
\end{equation*}

Let $\tau(m) =\vol_{k-2}(\QQ_1(\alpha, \beta) \cap \{\vec{x}:  \vec{m} \cdot \vec{x}  = m\})$, so that \eqref{eq:Fubinied} becomes
\begin{equation}\label{eq:OurIntegral}
\vol_{k-1} \QQ_1(\alpha, \beta)
= \int_{\R} \tau( \vec{m}\cdot \bfell(s))\,ds. 
\end{equation}
Since $P$ is Hermitian and idempotent, 
\begin{equation*}
\norm{P \vec{m}}^2 = P \vec{m} \cdot P\vec{m} = \vec{m} \cdot P^{\T}P \vec{m} 
= \vec{m} \cdot P^2  \vec{m} =  \vec{m} \cdot P \vec{m}.
\end{equation*}
Applying the Pythagorean theorem to the $\vec m= P\vec m+\frac{\vec m\cdot \vec n}{\norm{\vec n}}\vec{n}$ ensures that
\begin{equation*}
\norm{\vec{m}}^2 = \norm{P\vec{m}}^2+\frac{ (\vec{m}\cdot \vec{n})^2 }{\norm{\vec{n}}^2},
\end{equation*}
so
\begin{equation*}
\norm{P\vec{m}}^2 
= \norm{\vec{m}}^2-\frac{ (\vec{m}\cdot \vec{n})^2 }{\norm{\vec{n}}^2} 
= \frac{\norm{\vec{m}}^2\norm{\vec{n}}^2-(\vec{m} \cdot \vec{n})^2}{\norm{\vec{n}}^2}. 
\end{equation*}
Thus,
\begin{align*}
     \vec{m} \cdot \bfell(s) 
     &= \vec{m} \cdot \bigg( \frac{\vec{n}}{\norm{\vec{n}}^2}+s\frac{P\vec{m}}{\norm{P\vec{m}}}\bigg)  
    =  \frac{\vec{m} \cdot \vec{n}}{\norm{\vec{n}}^2}  + s \frac{\vec{m} \cdot (P\vec{m})}{\norm{P\vec{m}}}    \\
    &=\frac{ \vec{m} \cdot \vec{n}}{\norm{\vec{n}}^2}+\frac{s\norm{P\vec{m}}^2}{\norm{P\vec{m}}}
    =\frac{ \vec{m} \cdot \vec{n}}{\norm{\vec{n}}^2}+s\norm{P\vec{m}}\\
    &=\frac{ \vec{m} \cdot \vec{n}}{\norm{\vec{n}}^2}+s\frac{\sqrt{\norm{\vec{m}}^2\norm{\vec{n}}^2-(\vec{m}\cdot \vec{n})^2}}{\norm{\vec{n}}}.
\end{align*}
Substitute $t = \vec{m} \cdot \bfell(s)$  into \eqref{eq:OurIntegral} and obtain
\begin{equation}\label{eq:PlugTau}
\vol_{k-1}\QQ_1(\alpha, \beta)
= \frac{\norm{\vec{n}}}{\sqrt{\norm{\vec{m}}^2\norm{\vec{n}}^2-( \vec{m} \cdot \vec{n})^2}}
\int_{\R}  \tau(t)\, dt. 
\end{equation}

For $t\in [\alpha,\beta]$, 
the intersection of $\QQ_1(\alpha, \beta)$ with $ \vec{m} \cdot \vec{x}  = t$ is the variable polyope $\Pi_{A}(\twovectorsmall{t}{1})$,
defined by \eqref{eq:VariablePolytope}.
For $t\notin [\alpha,\beta]$ the intersection is empty. With the computation
\begin{equation*}
    \det(A A^{\T}) 
    =\det\big( \twovectorsmall{ \vec{m}^{\T} }{ \vec{n}^{\T}} \tworowvector{ \vec{m} }{ \vec{n} } \big)    
    =\det\begin{bmatrix}  \vec{m} \cdot \vec{m} &  \vec{m} \cdot \vec{n} \\  \vec{n} \cdot \vec{m}  &  \vec{n} \cdot \vec{n}\end{bmatrix}
    =\norm{\vec{m}}^2\norm{\vec{n}}^2- (\vec{m} \cdot \vec{n})^2,
\end{equation*}
we see that \eqref{eq:SplineVol} provides
\begin{equation*}
\tau(t)=\chi_{[\alpha,\beta]}(t) T_A\big( \twovectorsmall{ t }{ 1} \big) \sqrt{\norm{\vec{m}}^2\norm{\vec{n}}^2 - (\vec{m} \cdot \vec{n})^2}.
\end{equation*}
Plug this into \eqref{eq:PlugTau} and obtain
\begin{align*}
&\vol_{k-1}\QQ_1(\alpha, \beta) \\
&\qquad= \frac{\norm{\vec{n}}}{\sqrt{\norm{\vec{m}}^2\norm{\vec{n}}^2-( \vec{m} \cdot \vec{n})^2}}
\int_\R \chi_{[\alpha,\beta]}(t) T_A\big( \twovectorsmall{ t }{ 1} \big)
\sqrt{\norm{\vec{m}}^2\norm{\vec{n}}^2 -(\vec{m} \cdot \vec{n})^2}\, dt\\
&\qquad =\norm{\vec{n}}\int_{\alpha}^{\beta} 
T_A\big( \twovectorsmall{ t }{ 1} \big)  \,dt.\qedhere
\end{align*}
\end{proof}

\subsection{Conclusion of the proof of Theorem \ref{Theorem:A}}\label{Subsection:FinalConclude}
We are now ready to conclude the proof of Theorem \ref{Theorem:A}.
The next step in this endeavor is establishing the desired inequality \eqref{eq:Concrete}.
Once that is done, we use tools from probability theory to obtain the asymptotic
result \eqref{eq:FancyStats}, which completes the proof.

\begin{proof}[Pf.~of Theorem \ref{Theorem:A}]
Let $\K=\{\vec{x}\in \R^k:  \vec{n} \cdot  \vec{x}  = n\}$.
Since
\begin{equation*}
\QQ_n(\alpha, \beta)=\{\vec{x}\in \R_{\geq 0}^k: \text{$\vec{n} \cdot \vec{x} = n$ and $\alpha n \leq  \vec{m} \cdot \vec{x}  \leq \beta n$}\}
\end{equation*}
is a convex subset of the affine $(k-1)$-dimensional subspace $\K \subset \R^n$, dilating it by $s$, which yields a set in another affine $(k-1)$-dimensional subspace, scales its volume by a factor of $s^{k-1}$. 
Since $\vec{n} \in \Z_{>0}^k$ the origin is not in the convex hull of the columns of $A$, so
Lemma \ref{Lemma:Cavalieri} yields
\begin{align}
\vol_{k-1} \QQ_n(\alpha, \beta)
&=n^{k-1}\vol_{k-1}(\QQ_n(\alpha, \beta)/n) \nonumber\\
&= n^{k-1}\vol_{k-1}(\{\vec{x}\in \R_{\geq 0}^k:
\text{$\vec{n} \cdot (n\vec{x})=n$ and $\alpha n \leq  \vec{m} \cdot (n\vec{x})  \leq \beta n$} \} ) \nonumber\\
&=n^{k-1}\vol_{k-1}(\{\vec{x}\in \R_{\geq 0}^k: \text{$\vec{n} \cdot \vec{x}=1$ and $\alpha \leq  \vec{m} \cdot \vec{x}  \leq \beta$} \} ) \nonumber\\
&=n^{k-1}\vol_{k-1}(\QQ_1(\alpha,\beta))\nonumber\\
&=n^{k-1} \norm{\vec{n}}\int_{\alpha}^{\beta} T_A\big( \twovectorsmall{t}{1} \big)\,dt \nonumber\\
&= n^{k-1} \norm{\vec{n}}\int_{\alpha}^{\beta} \frac{M\big( t ; \frac{m_1}{n_1}, \frac{m_2}{n_2}, \ldots, \frac{m_k}{n_k} \big)}{(k-1)!n_1n_2\cdots n_k} \,dt
\label{eq:VolQk1}
\end{align}
by Theorem \ref{Theorem:BSplineConnection2}.
Next Lemma \ref{Lemma:AlmostDone} ensures that
\begin{equation*}
\left| | \QQ_n(\alpha, \beta) \cap \Z^k| - \frac{\vol_{k-1}\QQ_n(\alpha, \beta)}{\norm{\vec{n}}} \right| \leq n^{k-2}E_1(k),
\end{equation*}
in which $E_1(k)$ is given by \eqref{eq:E1}.
Substitute \eqref{eq:VolQk1} and
\begin{equation*}
|\QQ_n(\alpha, \beta) \cap \Z^k| = \big|\!\{\vec x\in \ZZ_S(n): \vec{m} \cdot\vec x \in [\alpha n, \beta n]\} \!\big|  
\end{equation*}
into the previous inequality and obtain the desired inequality \eqref{eq:Concrete}.

Let us briefly consider the special case where $\vec{m}$ is the all-ones vector, $\alpha = -\infty$, and $\beta = +\infty$.
Then $\vec{m},\vec{n}$ are linearly independent
since $\gcd(n_1,n_2,\ldots,n_k) = 1$, so \eqref{eq:Concrete} provides the explicit estimate
\begin{equation}\label{eq:ExplicitZ}
\left\lvert |\ZZ_S(n)| - 
\frac{n^{k-1}}{(k-1)!n_1n_2\cdots n_k}
  \right \rvert \leq n^{k-2}E_1(k).
\end{equation}
Thus, 
\begin{equation}\label{eq:ZAsymptotic}
| \ZZ_S(n)| \sim \frac{n^{k-1}}{(k-1)!n_1n_2\dots n_k},
\end{equation}
in which $\sim$ denotes asymptotic equivalence as $n \to \infty$.

We return now to the general situation armed with this asymptotic approximation.
The final stages of the proof resemble those of \cite[Thm.~1]{semigroupsII}, wherein the exposition is more detailed and thorough.
Observe that \eqref{eq:Concrete} and \eqref{eq:ZAsymptotic} imply
\begin{equation}\label{eq:MeasuresConverge}
\lim_{n\to\infty}
\frac{ \big|\{\vec x\in \ZZ_S(n): \vec{m} \cdot(\vec{x}/n) \in [\alpha , \beta ]\} \big| }{ | \ZZ_S(n)| }
=
 \int_{\alpha}^{\beta} M\Big(t; \frac{m_1}{n_1}, \frac{m_2}{n_2}, \ldots, \frac{m_k}{n_k} \Big)\,dt
\end{equation}
for all $-\infty \leq \alpha \leq \beta \leq \infty$.
For $n \in \Z_{>0}$, consider the probability measures 
\begin{equation}\label{eq:nun}
\nu_n = \frac{1}{|\ZZ_S(n)|}\sum_{\vec{x} \in \ZZ_S(n)} \delta_{ \frac{\vec{m}\cdot \vec{x} }{n} },
\end{equation}
in which $\delta_x$ denotes the point mass at $x$. 
Let $\nu$ be the absolutely continuous probability measure on $\R$ whose Radon--Nikodym derivative with respect to Lebesgue measure is the probability density
$M(t; \frac{m_1}{n_1}, \frac{m_2}{n_2}, \ldots, \frac{m_k}{n_k} )$.
Then \eqref{eq:MeasuresConverge} ensures that $\nu_n \to \nu$ weakly  \cite[Thm.~25.8]{Billingsley}.
Weak convergence provides
\begin{align*}
\lim_{n\to\infty}  \frac{1}{ |\ZZ_S(n)|} \sum_{\vec{x} \in \ZZ_S(n)} f\Big( \frac{ \vec{m} \cdot \vec{x} }{n}\Big)
&= \lim_{n\to\infty} \int_{-\infty}^{\infty} f(t) \,d\nu_n(t) \\
&= \int_{-\infty}^{\infty} f(t) \,d\nu(t) \\
&= \int_{-\infty}^{\infty} f(t) M\Big(t; \frac{m_1}{n_1}, \frac{m_2}{n_2}, \ldots, \frac{m_k}{n_k} \Big)\,dt
\end{align*}
for any bounded continuous function $f:\R\to\C$. 
\end{proof}

\section{Theorem \ref{Theorem:B}: An explicit bound on the difference between the vector partition function and multivariate truncated power}\label{Section:TT} 
The vector partition function $t_A(\vec{b})$, defined by \eqref{eq:PartitionFunction}, 
counts the number of integer points in the variable polytope $\Pi_{A}(\vec{b})$, defined by \eqref{eq:VariablePolytope}. 
The multivariate truncated power $T_A(\vec{b})$, defined in Subsection \ref{Subsection:Truncated}, is the volume of $\Pi_{A}(\vec{b})$.  The next theorem shows that we can approximate one quantity with the other with an error bound that is completely explicit; that is, there are no implied constants left unspecified.  

In our context, the conclusion of \ref{Theorem:B} should be understood as telling us that the number of factorizations of $n$ having weighted factorization length $m$ can be approximated by an evaluation of a multivariate truncated power. Taking the sum of this result over a range of values for $m$ and using an integral to approximate the sum of values of the multivariate truncated power gives a second way to derive the asymptotic distribution of weighted factorization lengths. This is exactly what we do in the proof of Theorem \ref{Theorem:C}. We state Theorem \ref{Theorem:B} in terms of vector partition functions and multivariate truncated powers since it may be of interest to those who work with vector partition functions and to hint at possible generalizations of the result. Corollary \ref{Corollary:B} restates the result using weighted factorization lengths and B-splines.

\begin{thmx}\label{Theorem:B}
Let $\vec{n} = [n_i] \in \Z_{>0}^k$ and $\vec{m} =[m_i] \in \Z^k$, and suppose $A \Z^k = \Z^2$, in which $A = \tworowvector{ \vec{m}}{\vec{n}}^{\T} \in \M_{2\times k}(\Z)$.
For each $\vec{b} \in \Z^2$,
\begin{equation*}
 |t_A(\vec{b})-T_A(\vec{b}) | \leq n^{k-3}E_2(k),
\end{equation*}
in which
\begin{equation}\label{eq:E2}
E_2(k)
= 8^{k-3} (k-2)^{\frac{3k^2-7k-3}{2}}.
\end{equation}
\end{thmx}

\begin{proof}
Fix $\vec{b} = \twovectorsmall{m}{n} \in \Z^2$.
Since $A\Z^k =\Z^2$ there exists some $\vec{p}\in \Z^k$ such that 
$\vec{m} \cdot \vec{p} = m$ and $\vec{n} \cdot \vec{p} = n$.  Then $\Pi_{A}(\vec{b})$ is contained in the affine subspace
\begin{equation*}
\h = \{\vec{x}\in \R^k: \text{$(\vec{x}-\vec{p}) \cdot  \vec{m}  = 0$ and $(\vec{x}-\vec{p}) \cdot \vec{n}  = 0$}\}.
\end{equation*}
Consider the translate 
\begin{equation*}
\h'= \{\vec{x}\in \R^k:  \text{$\vec{x} \cdot  \vec{m}  = 0$ and $\vec{x} \cdot \vec{n}  = 0$}\}
\end{equation*}
of $\h$ and define $\phi:\h\to \h'$ by $\phi(\vec{x} )= \vec{x} - \vec{p}$. Since $\phi$ is translation by an integer vector, 
$\phi(\Pi_{A}(\vec{b}))$ has the same $(k-2)$-dimensional volume, diameter, and number of integer points as $\Pi_{A}(\vec{b})$.
Define
\begin{equation*}
\Lambda = A^{\T} \Z^2 = \tworowvector{ \vec{m}}{ \vec{n}}\Z^2,
\end{equation*}
which is an integer sublattice of $\Z^k$.
Since $A\Z^k=\Z^2$, Lemma \ref{Lemma:Complete} ensures that $\Lambda$ is primitive.
Moreover, $\Lambda^{\perp} = \h'\cap \Z^k$, so Theorem \ref{Theorem:BrillGordan} yields
\begin{equation*}
\det(\h'\cap \Z^k)=
\det(\Lambda^{\perp}) 
= \det\Lambda 
= \sqrt{\det\big((A^{\T})^{\T} (A^{\T})\big)}= \sqrt{\det(A A^{\T})} .
\end{equation*}
Then \eqref{eq:SplineVol} provides
\begin{equation*}
\frac{\vol_{k-2}(\Pi_A(\vec{b}))}{\det(\h'\cap \Z^k)} = \frac{\sqrt{\det(A A^{\T})} T_A(\vec{b})}{\sqrt{\det(A A^{\T})}}=T_A(\vec{b}). 
\end{equation*}

Since $\Pi_{A}(\vec{b})\subseteq \Pi_{\vec{n}^{\T}}(n)$, Lemma \ref{Lemma:InBall} tells us that 
\begin{equation*}
\diam( \phi(\Pi_{A}(\vec{b}))) \leq \sqrt{2}n.
\end{equation*}
Consequently, Theorem \ref{Theorem:BodyCount} with $s=k-2$ implies that
\begin{align*} \left | t_A(\vec{b})-T_A(\vec{b}) \right | 
&=\bigg| \big| \Pi_A(\vec{b})\cap \Z^k\big| - \frac{\vol_{k-2}(\Pi_A(\vec{b}))}{\det(\h'\cap \Z^k)} \bigg| \\
&=\bigg| \big| \phi(\Pi_A(\vec{b}))\cap \Z^k\big| - \frac{\vol_{k-2}(\Pi_A(\vec{b}))}{\det(\h'\cap \Z^k)} \bigg| \\
&\leq(4 \sqrt{2}\diam( \phi(\Pi_{A}(\vec{b}))) )^{k-2-1}(k-2)^{\frac{3(k-2)^2+5(k-2)-5}{2}}\\
&\leq \big(8n\big)^{k-3}(k-2)^{\frac{3k^2-12k+12+5k-10-5}{2}}\\
&=n^{k-3} 8^{k-3} (k-2)^{\frac{3k^2-7k-3}{2}}\\
&=n^{k-3}E_{2}(k). \qedhere
\end{align*}
\end{proof}

\begin{remark}
The hypothesis $A\Z^k=\Z^2$ ensures that $\gcd(n_1,n_2,\ldots,n_k)=1$
and that $\vec{m},\vec{n}$ are linearly independent.
Lemma \ref{Lemma:Complete} and Corollary \ref{Corollary:Delta} provide
simple conditions that ensure $A \Z^k = \Z^2$. Examples \ref{Example:Fibonacci} and \ref{Example:Real} illustrate the ways that Theorem \ref{Theorem:B} can fail when $A\Z^k\neq \Z^2$ and when $A$ is not an integer matrix.
\end{remark}

\begin{remark}
Going back to \eqref{eq:CouldDo} reveals that the  $8^{k-3}$ in \eqref{eq:E2} can be replaced by 
\begin{equation*}
(4\sqrt{2}\max_{i\neq j} (n_i^{-2} + n_j^{-2})^{1/2})^{k-3}.  
\end{equation*}
We prefer \eqref{eq:E1} as is, since it depends only upon $k$.
\end{remark}

The following corollary restates Theorem \ref{Theorem:B} in the language of Theorem \ref{Theorem:A}. 
\begin{corollary}\label{Corollary:B}
 Let $\vec{n} = [n_i] \in \Z_{>0}^k$ and $\vec{m} =[m_i] \in \Z^k$, and suppose $A \Z^k = \Z^2$, in which $A = \tworowvector{ \vec{m}}{\vec{n}}^{\T} \in \M_{2\times k}(\Z)$. Set $S=\langle n_1, n_2, \dots, n_k \rangle \subseteq \Z_{\geq 0}.$
For any $n\in \Z_{>0}$ and $m\in \Z,$
\begin{equation*}
	\Big||\{\vec x\in \ZZ_S(n):\vec m \cdot x = m\}|-\frac{n^{k-2}}{(k-1)!n_1n_2\cdots n_k}M(\tfrac{m}{n};\tfrac{m_1}{n_1},\tfrac{m_2}{n_2},\cdots, \tfrac{m_k}{n_k}) \Big| \leq n^{k-3}E_2(k),
\end{equation*}
in which
\begin{equation}\label{eq:E2}
	E_2(k)
	= 8^{k-3} (k-2)^{\frac{3k^2-7k-3}{2}}.
\end{equation}
\end{corollary}
\begin{proof}
	The partition function $t_A(\twovectorsmall{m}{n})$  equals  $|\{\vec{x}\in \Z^k:  \vec{m} \cdot \vec{x} = m, \vec{n}\cdot \vec{x} = n\}$|, that is, the number of factorizations of $n$ in $S$ with $\vec{m}$-weighted factorization length  $m$. Theorem \ref{Theorem:BSplineConnection2} says that $T_A(\twovectorsmall{m}{n})=\frac{n^{k-2}}{(k-1)!n_1n_2\cdots n_k}M(\tfrac{m}{n};\tfrac{m_1}{n_1},\tfrac{m_2}{n_2},\cdots, \tfrac{m_k}{n_k})$. With these translations, the corollary is equivalent to Theorem \ref{Theorem:B}.
\end{proof}


\section{Theorem \ref{Theorem:C}: Bounding convergence of statistics}\label{Section:TheoremLL}

Our final theorem is version of the asymptotic statement \eqref{eq:FancyStats} in Theorem \ref{Theorem:A}, this time
with explicit error bounds.  Although this comes at the cost of more restrictive hypotheses than in Theorem \ref{Theorem:A}, in specific applications it is not difficult to compute the desired upper bound.   The rest of this section is devoted to the proof of the next result.

\begin{thmx}\label{Theorem:C}
Let $\vec{m} =[m_i] \in \Z^k$ and $\vec{n} = [n_i] \in \Z_{>0}^k$, and let $S = \semigroup{n_1,n_2,\ldots,n_k}$.
\begin{enumerate}\addtolength{\itemsep}{5pt}
\item Suppose $A \Z^k = \Z^2$, in which $A = \tworowvector{ \vec{m}}{\vec{n}}^{\T} \in \M_{2\times k}(\Z)$.
\item Let $n\geq 1$ and $\alpha,\beta\in \R \cup \{\pm \infty\}$ such that $\frac{1}{n} \leq \beta - \alpha$.
\item Let $f:[\alpha, \beta] \to \R$ be continuous with $|f(x)| \leq C_1$ for $x\in [\alpha, \beta]$.
\item For $x\in [\alpha n,\beta n]$, suppose $|f(x)T_A\big( \twovectorsmall{x}{1} \big)| \leq C_2$ and has Lipschitz constant $C_3$.
\end{enumerate}
Then, 
\begin{equation*}
\left| \sum_{\substack{\vec{x}\in \ZZ_S(n) \\ \vec{m} \cdot (\vec{x}/n) \in [\alpha,\beta]}}f\left(\frac{\vec{m} \cdot \vec{x}}{n}\right) 
- 
n^{k-1}\int_{\alpha}^\beta f(t) \frac{M\big(t;\frac{m_1}{n_1},\frac{m_2}{n_2},\ldots, \frac{m_k}{n_k}\big)}{(k-1)!n_1n_2\dots n_k}  \,dt \right| 
\leq E_{3}(n,k),
\end{equation*}
in which $M\big(t; \frac{m_1}{n_1}, \frac{m_2}{n_2}, \ldots, \frac{m_k}{n_k} \big)$ is a Curry--Schoenberg B-spline and
\begin{equation}\label{eq:E3}
\begin{split}
E_{3}(n,k)
&=
\left((\beta-\alpha)\left(C_1 8^{k-3} (k-2)^{\frac{3k^2-7k-3}{2}}+C_3\right)+2C_2\right)n^{k-2} \\
&\qquad + C_1 8^{k-3} (k-2)^{\frac{3k^2-7k-3}{2}}n^{k-3}.
\end{split}
\end{equation}
\end{thmx}

\subsection{A technical lemma}\label{Subsection:Technical}
At the heart of the proof of Theorem \ref{Theorem:C} is the following technical lemma, which generalizes \cite[Lem.~23]{semigroupsIV}. It provides an explicit estimate, with no unspecified implied constants, of the difference between a certain discrete sum and a corresponding integral. The details are considerable, but elementary.  However, since the point of Theorem \ref{Theorem:C} is the explicit nature of the error term we feel compelled to go through the argument step-by-step. 
 
\begin{proposition}\label{Proposition:Technical}
Let $\alpha,\beta \in \R \cup \{\pm \infty\}$ with $\alpha<\beta$. 
Let $n \in \Z_{>0}$ satisfy $\tfrac{1}{n}\leq \beta-\alpha$ and let $g_n:\Z \to \R$ and $h:[\alpha,\beta]\to \R$ satisfy
\begin{enumerate}\addtolength{\itemsep}{5pt}
\item $| g_n(\ell)-n^{k-2}h(\ell/n) | < K_1n^{k-3}$ for $\ell \in \Z\cap [\alpha n, \beta n]$,
\item $|h(x)|$ is Lipschitz on $[\alpha, \beta]$ with Lipschitz constant $K_2$, and
\item  $|h(x)| \leq K_3$ for $x \in [\alpha, \beta]$.
\end{enumerate}
Then
\begin{equation*}
\abs{\frac{\sum_{\ell\in \Z\cap [\alpha n, \beta n]}g_n(\ell)}{n^{k-1}} - \int_{\alpha}^{\beta} h(x)\,dx}
\leq \frac{(\beta-\alpha)(K_1+K_2)+2K_3}{n} + \frac{K_1}{n^2}.
\end{equation*}
\end{proposition}

\begin{proof}
The inequalities below are used implicitly throughout the proof:
\begin{equation*}
\alpha\leq \frac{\ceil{\alpha n}}{n} \leq \beta \leq \frac{\floor{\beta n +1}}{n}
\quad \text{and} \quad 
\alpha n \leq \floor{\beta n} .
\end{equation*}
Since the definitions of the ceiling and floor functions ensure that
$\alpha n\leq \ceil{\alpha n}$ and $\beta n \leq \floor{\beta n +1}$,
respectively, we just need to check that 
$\ceil{\alpha n} \leq \beta n$ and $\alpha n \leq \floor{\beta n}$.
These inequalities follow from the hypothesis $\frac{1}{n} \leq \beta - \alpha$ since
\begin{equation*}
\ceil{\alpha n}  < \alpha n+1 \leq (\beta-\tfrac{1}{n})n+1 =\beta n
\end{equation*}
and
\begin{equation*}
    \floor{\beta n} > \beta n-1 \geq (\alpha+\tfrac{1}{n})n-1 =\alpha n. 
\end{equation*}

Define $$H_n(x)=ng_n(\floor{xn})$$ and observe that
\begin{equation*}
\int_{\frac{\ceil{\alpha n}}{n}}^{\frac{\floor{\beta n+1}}{n}} \!\! H_n(x)\,dx
=n \int_{\frac{\ceil{\alpha n}}{n}}^{\frac{\floor{\beta n+1}}{n}} \!\! g_n(\floor{xn})\,dx
=\int_{\ceil{\alpha n}}^{\floor{\beta n}+1} \!\! g_n(\floor{y})\,dy
= \!\!\!\! \sum_{\ell\in \Z\cap [\alpha n, \beta n]}g_n(\ell).
\end{equation*}
For $\ell \in \Z\cap [\alpha n, \beta n] = \Z\cap \big[ \ceil{\alpha n} , \floor{\beta n} \big]$, condition (a) ensures that
\begin{equation*}
    \bigg| \frac{H_n(\frac{\ell}{n})}{n^{k-1}}-h(\tfrac{\ell}{n}) \bigg| 
    =\bigg|  \frac{ng_n(\floor{n\frac{\ell}{n} })}{n^{k-1}}-h(\tfrac{\ell}{n}) \bigg| 
    =\bigg|  \frac{g_n(\ell) - n^{k-2}h(\tfrac{\ell}{n})}{n^{k-2}} \bigg| 
    < \frac{K_1}{n}.
\end{equation*}
Condition (b) and the above with $\ell = \floor{xn}$ for $x\in \big[ \frac{\ceil{\alpha n}}{n}, \beta\big]$ provide
\begin{equation*}
\abs{\frac{H_n(x)}{n^{k-1}}-h(x)}
\leq \abs{\frac{H_n \big( \tfrac{\floor{xn}}{n} \big) }{n^{k-1}} -h\bigg(\frac{\floor{xn}}{n}\bigg)}+\abs{h\bigg(\frac{\floor{xn}}{n}\bigg) - h(x)}
\leq \frac{K_1+K_2}{n}
\end{equation*}
since $H_n(\floor{xn}/n) = H_n(x)$.
Thus,
\begin{align*}
&\abs{\frac{\sum_{\ell\in \Z\cap [\alpha n, \beta n]}g_n(\ell)}{n^{k-1}} - \int_{\alpha}^{\beta} h(x)\,dx} \\
&\qquad =\abs{\int_{\frac{\ceil{\alpha n}}{n}}^{\frac{\floor{\beta n +1}}{n}} \frac{H_n(x)}{n^{k-1}}\,dx - \int_{\alpha}^{\beta} h(x)\,dx}\\
&\qquad =\bigg| \int_{\frac{\ceil{\alpha n}}{n}}^{\beta}  \frac{H_n(x)}{n^{k-1}}-h(x) \, dx -\int_{\alpha}^{\frac{\ceil{\alpha n}}{n}}h(x)\,dx+\int_{\beta}^{\frac{\lfloor\beta n+1 \rfloor}{n}}  \frac{H_n(x)}{n^{k-1}}\,dx \bigg| \\
&\qquad \leq \bigg| \int_{\frac{\ceil{\alpha n}}{n}}^{\beta}\frac{H_n(x)}{n^{k-1}}-h(x)\,dx \bigg| 
+\bigg| \int_{\alpha}^{\frac{\ceil{\alpha n}}{n}}h(x)\,dx \bigg|
    + \bigg| \int_{\beta}^{\frac{\lfloor \beta n+1 \rfloor}{n}}  \frac{ g_n(\floor{xn})}{n^{k-2}} \,dx \bigg|\\
&\qquad = \bigg| \int_{\frac{\ceil{\alpha n}}{n}}^{\beta}\frac{H_n(x)}{n^{k-1}}-h(x)\,dx \bigg| 
+\bigg|\int_{\alpha}^{\frac{\ceil{\alpha n}}{n}}h(x)\,dx\bigg|+
\bigg| g_n(\floor{\beta n}) \int_{\beta}^{\frac{\lfloor \beta n+1 \rfloor}{n}}  \frac{ dx}{n^{k-2}} \bigg|\\
&\qquad \leq \bigg| \int_{\frac{\ceil{\alpha n}}{n}}^{\beta}\frac{H_n(x)}{n^{k-1}}-h(x)\,dx \bigg| 
+\frac{K_3}{n} +  \frac{ g_n(\floor{\beta n})}{n^{k-1}} \\
&\qquad \leq \bigg| \int_{\frac{\ceil{\alpha n}}{n}}^{\beta}\frac{H_n(x)}{n^{k-1}}-h(x)\,dx \bigg| 
+\frac{K_3}{n}
+\frac{1}{n}h\Big(\frac{\floor{\beta n}}{n}\Big)+\frac{K_1}{n^{2}}\\
&\qquad \leq \int_{\frac{\ceil{\alpha n}}{n}}^{\beta}\bigg| \frac{H_n(x)}{n^{k-1}}-h(x) \bigg| \,dx +\frac{2K_3}{n}+\frac{K_1}{n^{2}}\\
&\qquad \leq \int_{\frac{\ceil{\alpha n}}{n}}^{\beta} \frac{K_1+K_2}{n}\,dx +\frac{2K_3}{n}+\frac{K_1}{n^{2}}\\
&\qquad =\left( \beta - \frac{\ceil{\alpha n}}{n}\right) \frac{K_1+K_2}{n}+\frac{2K_3}{n}+\frac{K_1}{n^{2}}\\
&\qquad \leq \frac{  (\beta-\alpha)  (K_1+K_2)+2K_3}{n} + \frac{K_1}{n^2}. \qedhere
\end{align*}
\end{proof}

\subsection{Proof of Theorem \ref{Theorem:C}}
By hypothesis, $A = \tworowvector{ \vec{m}}{\vec{n}}^{\T} \in \M_{2\times k}(\Z)$ satisfies $A \Z^k = \Z^2$, in which
$\vec{m} =[m_i] \in \Z^k$ and $\vec{n} = [n_i] \in \Z_{>0}^k$.  Furthermore, 
$\alpha,\beta \in \R\cup \{\pm \infty\}$ with $\alpha<\beta$ and we fix $n$ such that $\frac{1}{n} \leq \beta - \alpha$.
Let
\begin{equation*}
g_n(\ell)=t_A\big(\twovectorsmall{\ell}{n}\big)f(\ell/n)
\quad \text{and} \quad
h(x) = T_A \big( \twovectorsmall{x}{1} \big)f(x).
\end{equation*}
Then $h(x)$ is bounded in absolute value by $C_2$ and has Lipschitz constant $C_3$ for $x\in [\alpha n, \beta n]$.
Moreover, $f(x)$ is continuous and bounded by $C_1$ on $x\in [\alpha, \beta ]$.

Since $T_A(\vec{b})=\vol_{k-2}\Pi_A(\vec{b})$, we have $T_A(s\vec{b})=s^{k-2}T_A(\vec{b})$ for $s > 0$
(this also follows from Theorem \ref{Theorem:BSplineConnection2}). 
For $\ell\in \Z \cap [\alpha n, \beta n]$, Theorem \ref{Theorem:B} ensures that
\begin{align*}
    \abs{g_n(\ell)-n^{k-2}h(\ell/n)}
    &=\abs{t_A(\twovectorsmall{\ell}{n})f(\ell/n)-n^{k-2}T_A \big( \twovectorsmall{\ell/n}{1} \big)f(\ell/n)}\\
    &=|f(\ell/n)|\abs{t_A(\twovectorsmall{\ell}{n})-n^{k-2}T_A \big( \twovectorsmall{\ell/n}{1} \big)}\\
    &=|f(\ell/n)|\abs{t_A(\twovectorsmall{\ell}{n})-T_A \big( \twovectorsmall{\ell}{n} \big)}\\
    &\leq n^{k-3}C_1 E_2,
\end{align*}
in which $E_2 = E_2(k)$ is given by \eqref{eq:E2}.
Then Proposition \ref{Proposition:Technical} with $K_1=C_1E_2$, $K_2 = C_3$, and $K_3 = C_2$ yields 
\begin{equation}\label{eq:ThmCButWrong}
    \abs{\frac{\sum_{\ell\in \Z\cap [\alpha n, \beta n]}g_n(\ell)}{n^{k-1}} - \int_{\alpha}^{\beta} h(x)\,dx}
\leq \frac{(\beta-\alpha)(C_1E_2+C_3)+2C_2}{n} + \frac{C_1E_2}{n^2}.
\end{equation}
Next, the definitions of $g_n$, $t_A$, and $\ZZ_S(n)$ ensure that
\begin{align*}
    \sum_{\ell\in \Z\cap [\alpha n, \beta n]}g_n(\ell)&=\sum_{\ell\in \Z\cap [\alpha n, \beta n]}t_A\big(\twovectorsmall{\ell}{n}\big)f(\ell/n)\\
    &=\sum_{\ell\in \Z\cap [\alpha n, \beta n]} \big| \{\vec x\in \Z^k_{\geq 0}:A\vec x=\twovectorsmall{\ell}{n}\}\big|\,  f(\ell/n)\\
    &=\sum_{\substack{\vec{x}\in \ZZ_S(n) \\ \vec{m} \cdot \vec{x} \in [\alpha,\beta]}}f\left(\frac{\vec{m} \cdot \vec{x}}{n}\right).
\end{align*}
Apply Theorem \ref{Theorem:BSplineConnection2} to the second term on the left side of \eqref{eq:ThmCButWrong} and get
\begin{equation*}
\int_{\alpha}^{\beta} h(x)\,dx
=\int_{\alpha}^{\beta} T_A \big( \twovectorsmall{x}{1} \big)f(x)\,dx 
=\int_{\alpha}^\beta f(x)\frac{M(x;\frac{m_1}{n_1},\tfrac{m_2}{n_2},\dots, \frac{m_k}{n_k})}{(k-1)!n_1n_2\dots n_k} \,dx.
\end{equation*}
Now multiply \eqref{eq:ThmCButWrong} by $n^{k-1}$ and obtain
\begin{equation*}
\left\lvert\sum_{\substack{\vec{x}\in \ZZ_S(n) \\ \vec{m} \cdot \vec{x} \in [\alpha,\beta]}}f\left(\frac{\vec{m} \cdot \vec{x}}{n}\right) - n^{k-1}\int_{\alpha}^\beta f(x)\frac{M(x;\frac{m_1}{n_1},\tfrac{m_2}{n_2},\dots, \frac{m_k}{n_k})}{(k-1)!n_1n_2\dots n_k}\,dx\right\rvert \leq E_{3}(n,k),
\end{equation*}
in which $E_3(n,k)$ is given by \eqref{eq:E3}. \qed

\section{Examples and applications}\label{Section:Examples}

This section exhibits a host of examples and applications of Theorems \ref{Theorem:A}, \ref{Theorem:B}, and \ref{Theorem:C}.
In particular, these examples illuminate the asymptotic behavior of weighted factorization lengths in the context of numerical semigroups.
Not only does this recover the main results of \cite{semigroupsII, semigroupsIV}, we can also study new phenomena entirely.

We prefer to present our examples in a primarily visual format.  Although we could present tables of 
numerical data concerning explicit inequalities, these would not be so striking or appealing.
We first consider integral weights, since the more general setting of real weights requires a different
method of visualization.  

Let $S = \semigroup{n_1,n_2,\ldots,n_k}$ be a numerical semigroup, in which
the generators may be repeated and need not be minimally presented, and let $\vec{m} = (m_1,m_2,\ldots,m_k)\in \Z^k$.  
For each weighted factorization length $\ell = \vec{m} \cdot \vec{x}$ of $n = \vec{n} \cdot \vec{x} \in S$,
we place a  dot above $\ell/n$ at height equal to $n/|\ZZ_S(n)|$ times the multiplicity of $\ell$ 
as a weighted factorization length of $n$.\footnote{The captions in \cite[Figs.~2, 3, 5, 6]{semigroupsII} misstated the vertical normalization.} 
If no factorizations of weighted length $\ell$ exist, the dot height is $0$.
These normalizations ensure that
the dot plots ``converge'' to the corresponding B-spline.  More precisely,  the $n$th collection of dots
can be viewed as representing the singular probability measure $\nu_n$ from \eqref{eq:nun}. These measures converge
weakly to the displayed B-spline. On the other hand, when Theorem \ref{Theorem:B} holds it implies that the probability density functions $\mu_n = \frac{n}{|\ZZ_S(n)|}\sum_{\vec x \in  \ZZ_S(n)} \chi_{((\vec x\cdot \vec{m}-1)/n, \vec x\cdot \vec{m}/n]}$ converge pointwise to the same B-spline. In this way, one can think of Theorem \ref{Theorem:B} as a ``local limit theorem''. Every plot where the points plotted are each close to the B-spline plotted is thus an illustration of Theorem \ref{Theorem:B}.

    \begin{figure}
    \centering
    \begin{subfigure}{0.475\textwidth}
      \centering
      \includegraphics[width=\textwidth]{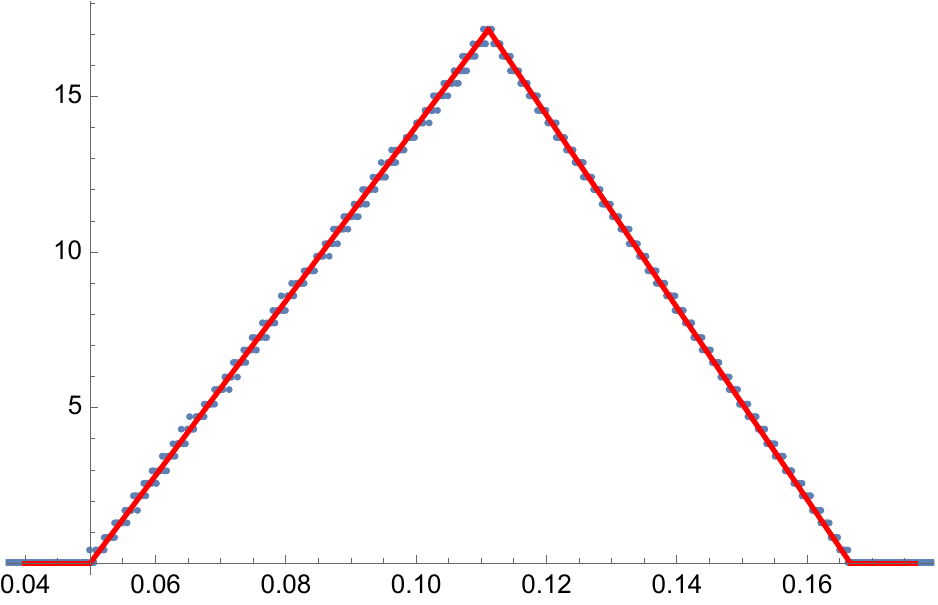}
      \caption{$\vec{m} = (1,1,1)$, $n = 5{,}000$}
      \label{Figure:NuggetLength}
    \end{subfigure}
    \hfill
    \begin{subfigure}{0.475\textwidth}
      \centering
      \includegraphics[width=\textwidth]{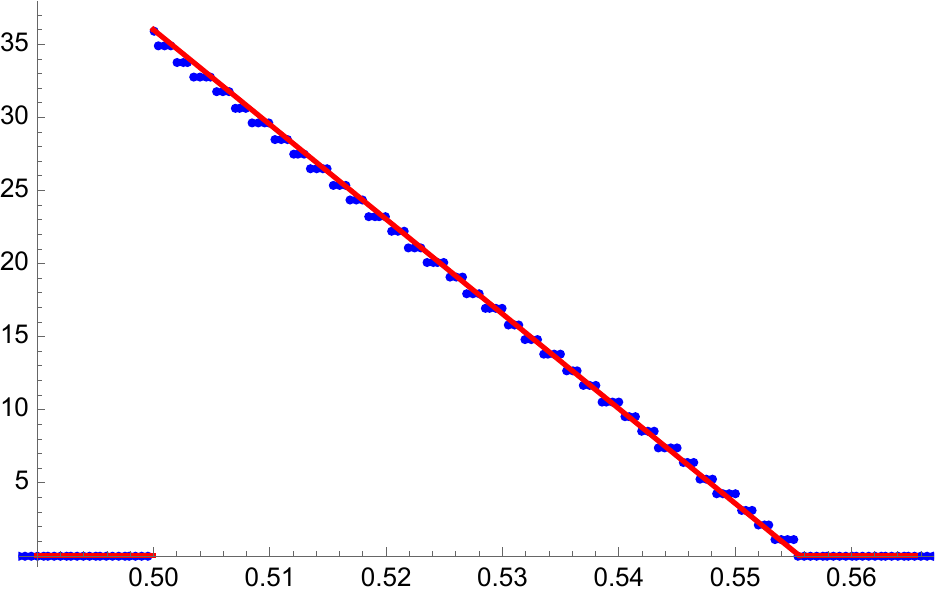}
      \caption{$\vec{m} = (3,5,10)$, $n=2{,}000$}
      \label{Figure:Nugget3510}
      \end{subfigure}
    \caption{Weighted factorization lengths in the McNugget semigroup $S = \semigroup{6,9,20}$ from Example \ref{Example:Nugget}. A dot appears above $\ell/n$ at height $n/|\ZZ_S(n)|$ times the multiplicity of $\ell$  as a weighted factorization length of $n$.}
    \label{Figure:NuggetBasic}
    \end{figure}
    
\begin{example}\label{Example:Nugget}
The \emph{McNugget semigroup} is $S=\semigroup{6,9,20}$, so named because Chicken McNuggets originally came in boxes of $6$, $9$, and $20$ pieces.
A factorization $n = 6x_1 + 9 x_2 + 20 x_3 = \vec{n} \cdot \vec{x}$, in which
$\vec{n} = (6,9,20)$ and $\vec{x} = (x_1,x_2,x_3)$, describes one way a cashier can respond to an order for exactly $n$ nuggets:
serve the customer $x_1$ boxes of six nuggets, $x_2$ boxes of nine nuggets, and $x_3$ boxes of twenty.
The factorization length $x_1 + x_2 + x_3 = \vec{x} \cdot \vec{1}$,
in which $\vec{1} = (1,1,1)$, is the total number of boxes used in filling this order.
The corresponding B-spline is (Figure \ref{Figure:NuggetLength})
\begin{equation*}
M(x;\tfrac{1}{20}, \tfrac{1}{9}, \tfrac{1}{6})=
\begin{cases}
 \frac{1080}{77} (20 x-1) & \frac{1}{20}\leq x<\frac{1}{9}, \\[3pt]
 -\frac{360}{7}  (6 x-1) & \frac{1}{9}\leq x < \frac{1}{6}, \\[3pt]
 0 & \text{otherwise}.
\end{cases}
\end{equation*}
The mean factorization length for $n=5{,}000$ is, to two decimal places, $546.66$, which is 
close to the predicted value $5{,}000 \int_{-\infty}^{\infty} t M(t;\tfrac{1}{20}, \tfrac{1}{9}, \tfrac{1}{6})\,dt = 546.30$.

We can consider weighted
factorization lengths as well.  Let $\vec{m} = (3,5,10)$, so the corresponding
weighted factorization length is $\vec{m} \cdot \vec{x} = 3x_1 + 5x_2 + 10x_3$.  In this setting, the associated
B-spline is (Figure \ref{Figure:Nugget3510})
\begin{equation*}
M(x; \tfrac{1}{2}, \tfrac{1}{2}, \tfrac{5}{9}) = 
\begin{cases}
 72 (5-9 x) & \text{if $\frac{1}{2}\leq x\leq \frac{5}{9}$},\\
 0 & \text{otherwise},
\end{cases}
\end{equation*}
from which we compute various statistics with $n=2{,}000$ 
\begin{equation*}
\begin{array}{c|cccc}
 & \text{mean} & \text{median} & \text{mode} & \text{st.~dev.} \\
\hline
\text{actual} & 1036.90 & 1032.00 & 1000 &  26.44\\
\text{predicted} &1037.04 & 1032.54& 1000 & 26.19\\
\end{array}
\end{equation*}
to two digits of accuracy.
For example, the mode is estimated as $n$ times the unique point (namely, $x= \frac{1}{2}$) 
where $M(x; \frac{1}{2}, \frac{1}{2}, \frac{5}{9})$ is maximized.
Similarly, the median is estimated as $n$ times the unique point
(namely, $x = \frac{1}{36} (20-\sqrt{2})$) such that $\int_{-\infty}^x M(x; \frac{1}{2}, \frac{1}{2}, \frac{5}{9})\,dx = \frac{1}{2}$.
For $n=2{,}000$, this yields $1032.54$.  These estimates are close to the actual values, obtained by computation from the discrete data.
\end{example}

\begin{example}\label{Example:Picnic}
A group of $n$ people walk into McDonalds and each buys two boxes of McNuggets each.
The overwhelmed staff fills orders by randomly choosing McNugget boxes with capacity $6,9$, or $20$ and giving two boxes to each person.  Each person can eat at most $30$ nuggets and they each require at least $15$ nuggets to be sated.  The group redistributes the nuggets amongst themselves when they reach a nearby park, attempting to make everyone satisfied without anyone being too full.  What is the probability that they have a happy picnic?

This is the dual problem of the classic factorization-length problem for chicken nuggets (Example \ref{Example:Nugget}).
Let $\vec{n}=(1,1,1)$, so $S = \semigroup{1,1,1}$, and let $\vec{m}=(6,9,20)$.  Although as a set, $S$ is just the set of nonnegative integers,
we regard it as a numerical semigroup generated by three $1$s of different colors.
The group of $n$ people takes $2n$ boxes of nuggets.  The number of ways to
serve $2n$ boxes of nuggets is the number $|\ZZ_S(2n)|$ of factorizations of $2n$
in $S$.  The group needs between $15n$ and $30n$ to have a happy picnic. 
The proportion of factorizations of $n$ that yield between $15n = 7.5(2n)$ and $30n = 15(2n)$ total nuggets is approximated by
the integral of a corresponding B-spline $M(x;6,9,20)$.  To three decimal places
\begin{equation*}
\int_{7.5}^{15} M(x;6,9,20)\,dx = 0.784.
\end{equation*}
If $n = 100$, we have $|\ZZ_S(2n)| = 20{,}300$ and of all these ways to serve the group,
$15{,}847$ of them provide between $1500$ and $3000$ nuggets, a proportion of $0.781$.
\end{example}

    \begin{figure}
    \centering
    \begin{subfigure}{0.475\textwidth}
      \centering
      \includegraphics[width=\textwidth]{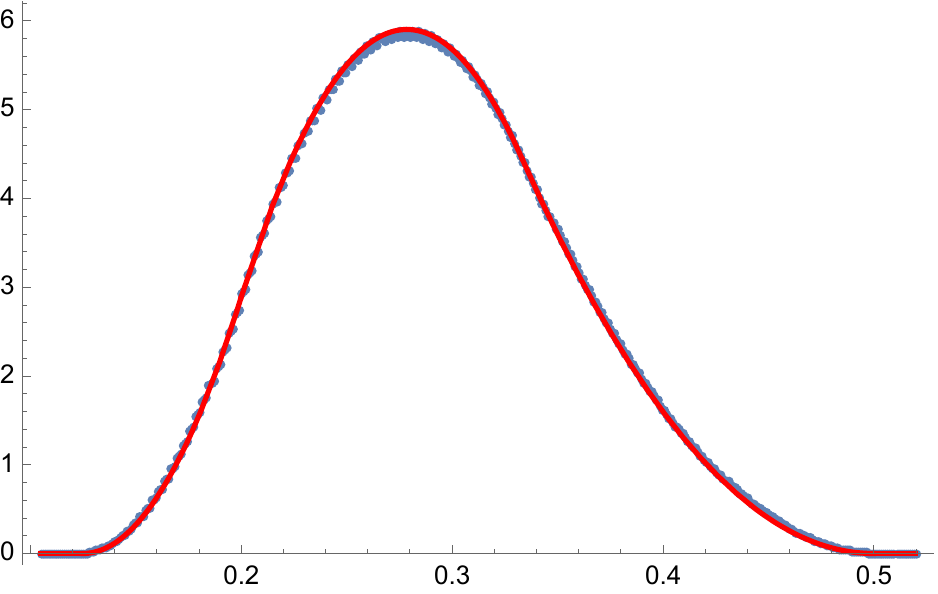}
      \caption{$\vec{m} = (1,1,1,1)$, $n=1{,}000$}
      \label{Figure:FibonacciLength}
    \end{subfigure}
    \hfill
    \begin{subfigure}{0.475\textwidth}
      \centering
      \includegraphics[width=\textwidth]{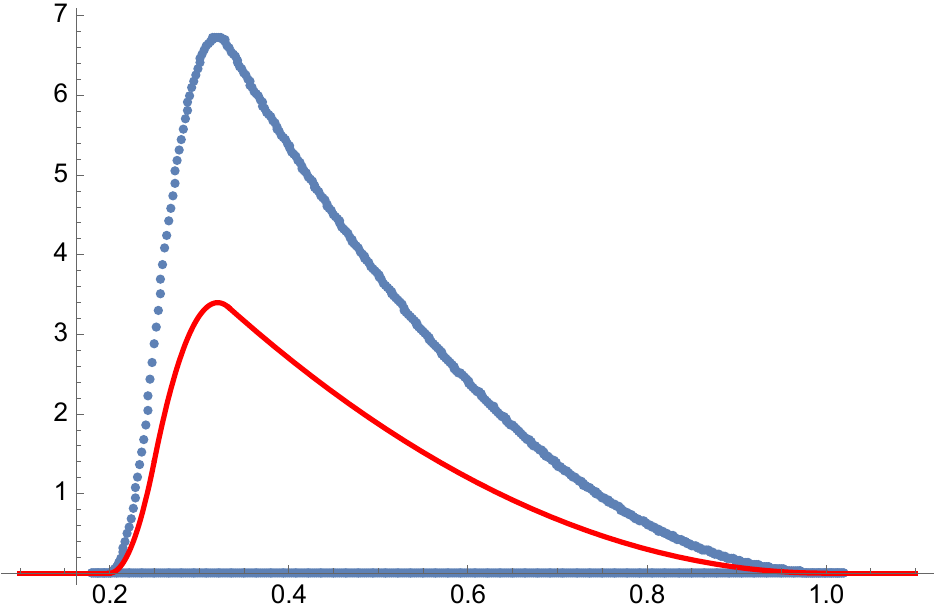}
      \caption{$\vec{m} = (2,1,1,2)$, $n=1{,}000$}
      \label{Figure:Fibonacci2112}
      \end{subfigure}
    \caption{Weighted factorization lengths in the non-minimally generated semigroup $S = \semigroup{2,3,5,8}$ from Example \ref{Example:Fibonacci}. A dot appears above $\ell/n$ at height  $n/|\ZZ_S(n)|$ times the multiplicity of $\ell$  as a weighted factorization length of $n$.}
    \label{Figure:Fibonacci}
    \end{figure}

\begin{example}\label{Example:Fibonacci}
Consider the numerical semigroup $S = \semigroup{2,3,5,8}$, which is presented with a non-minimal generating set.  Then $n=9$ has $3$ factorizations in $S$ with respect to the chosen presentation, namely $(0,3,0,0)$, $(2,0,1,0)$, and $(3,1,0,0)$,
which have factorization lengths $3$, $3$, and $4$, respectively.  With $\vec{n} = (2,3,5,7)$ and 
$\vec{m} = (1,1,1,1)$, we obtain factorization lengths; see Figure \ref{Figure:FibonacciLength}.

Suppose that we weight our factorizations so that the even generators count twice as much as odd generators.  The weight vector $\vec{m} = (2,1,1,2)$ leads to
\begin{equation*}
M(x;\tfrac{1}{5}, \tfrac{1}{4}, \tfrac{1}{3},1)=
\begin{cases}
 \frac{15}{2} (x^2-2 x+1) & \text{if $\frac{1}{3}\leq x\leq 1$}, \\[3pt]
 \frac{45}{2} (25 x^2-10 x+1) & \text{if $\frac{1}{5}\leq x<\frac{1}{4}$}, \\[3pt]
 -\frac{15}{2}  (53 x^2-34 x+5) & \text{if $\frac{1}{4}\leq x<\frac{1}{3}$}, \\[3pt]
 0 & \text{otherwise}.
\end{cases}
\end{equation*}
An interesting phenomenon occurs; see Figure \ref{Figure:Fibonacci2112}.  
If $n$ is even, we reduce $n = \vec{m} \cdot \vec{x}$ modulo $2$ and obtain
$0 \equiv x_2 + x_3 \pmod{2}$, so that $x_2 \equiv x_3 \pmod{2}$.
Therefore, $\ell = \vec{m} \cdot \vec{x}$ is never odd, so half of the dots in 
Figure \ref{Figure:Fibonacci2112} are on the horizontal axis. Clearly, the conclusion of Theorem \ref{Theorem:B} is false in this case. Indeed, $A = \big[\begin{smallmatrix} 2 & 1 & 1 & 2 \\ 2 & 3 & 5 & 8 \end{smallmatrix}\big]$ does not
satisfy the hypothesis $A \Z^4 = \Z^2$. This sort of behavior is exactly what that hypothesis is meant to capture.  Note that this also means that Theorem \ref{Theorem:C} cannot be applied.
This does not contradict 
Theorem \ref{Theorem:A} because the associated singular measure $\nu_n$ from \eqref{eq:nun}
contains half as many summands as it would without such a parity obstruction.
\end{example}

    \begin{figure}
    \centering
    \begin{subfigure}{0.475\textwidth}
      \centering
      \includegraphics[width=\textwidth]{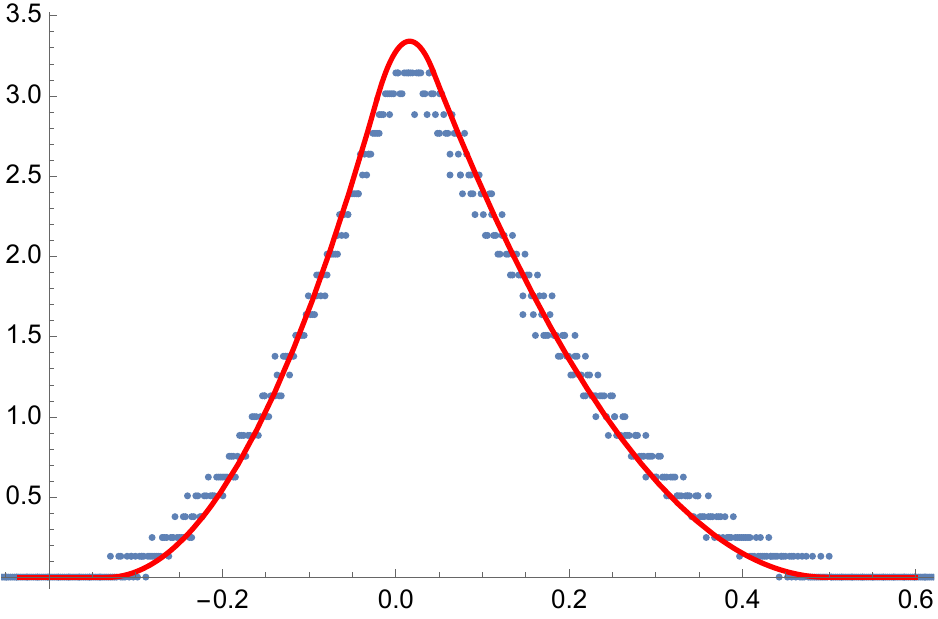}
      \caption{$\vec{m} = (1,1,-1,-1)$, $n=500$}
      \label{Figure:TwoMinus}
    \end{subfigure}
    \hfill
    \begin{subfigure}{0.475\textwidth}
      \centering
      \includegraphics[width=\textwidth]{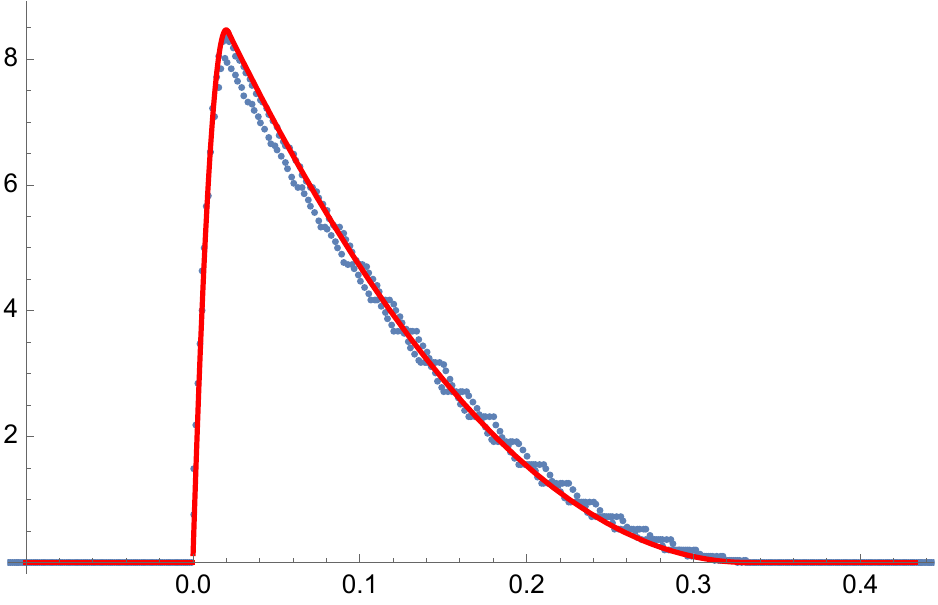}
      \caption{$\vec{m} = (0,0,1,1)$,  $n=1{,}000$}
      \label{Figure:TwoNone}
      \end{subfigure}
    \caption{Weighted factorization lengths in the non-minimally generated semigroup $S = \semigroup{2,23,3,47}$ from 
    Example \ref{Example:TwoSemis}. A dot appears above $\ell/n$ at height  $n/|\ZZ_S(n)|$ times the multiplicity of $\ell$  as a weighted factorization length of $n$.}
    \label{Figure:Two}
    \end{figure}

\begin{example}\label{Example:TwoSemis}
Combine $S_1 = \semigroup{2,23}$ and $S_2 = \semigroup{3,47}$ to construct
the (not minimally presented) semigroup $S = \semigroup{2,23,3,47}$, which has $\vec{n} = (2,23,3,47)$.

For the weight vector $\vec{m} = (1,1,-1,-1)$, the associated weighted factorization length $x_1 + x_2 - x_3 - x_4$ of an element of $S$
is the total number of generators used from $S_1$ minus the total number of generators used from $S_2$.
The corresponding B-spline is $M(x;-\frac{1}{3},-\frac{1}{47},\frac{1}{23},\frac{1}{2})$; see Figure \ref{Figure:TwoMinus}.  

For the weight vector $\vec{m} = (0,0,1,1)$, the associated weighted factorization
length $x_3+x_4$ of an element of $S$ ignores the generators from $S_1$.
The corresponding B-spline is $M(x; 0,0,\frac{1}{47},\frac{1}{3})$; see Figure \ref{Figure:TwoNone}.
\end{example}

    \begin{figure}
    \centering
    \begin{subfigure}{0.475\textwidth}
      \centering
      \includegraphics[width=\textwidth]{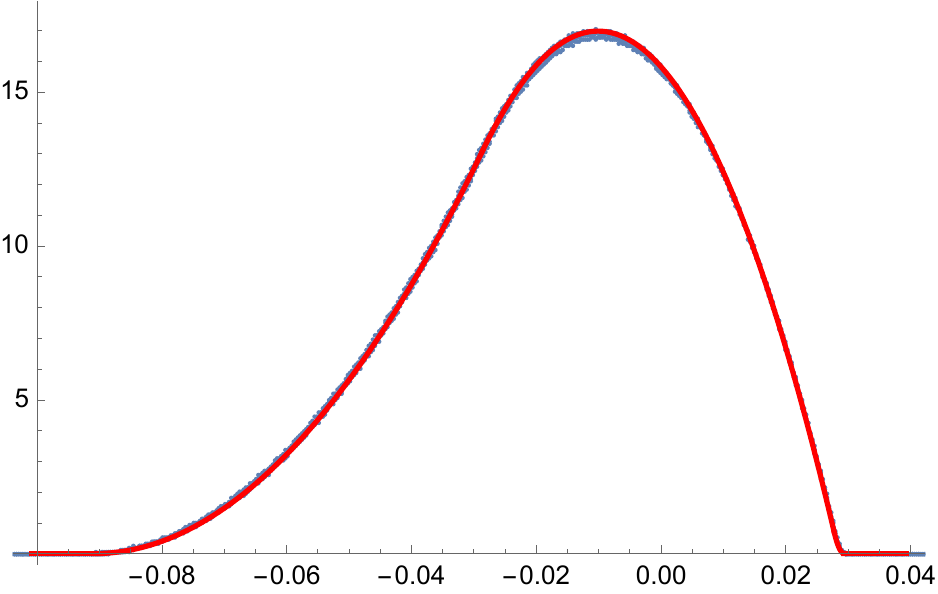}
      \caption{$\vec{m} = (-1, 1, -1,1)$,  $n=10{,}000$}
      \label{Figure:SillyOdd}
    \end{subfigure}
    \hfill
    \begin{subfigure}{0.475\textwidth}
      \centering
      \includegraphics[width=\textwidth]{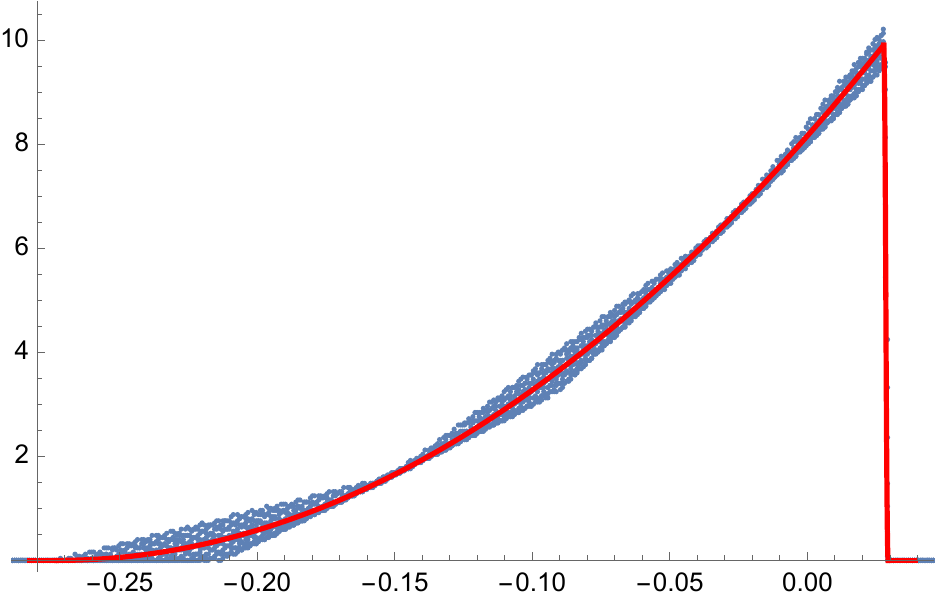}
      \caption{$\vec{m} = (-3,1,1,1)$,  $n=10{,}000$}
      \label{Figure:Silly36}
      \end{subfigure}
    \caption{Weighted factorization lengths in the numerical semigroup $S = \semigroup{11,34,35,36}$ from 
    Example \ref{Example:Silly}. A dot appears above $\ell/n$ at height  $n/|\ZZ_S(n)|$ times the multiplicity of $\ell$  as a weighted factorization length of $n$.}
    \label{Figure:Silly}
    \end{figure}

\begin{example}\label{Example:Silly}
Factorization lengths in the numerical semigroup $S = \semigroup{11,34,35,36}$
were studied in \cite[Ex.~8]{semigroupsII}.  We can now consider weighted factorization lengths on the same semigroup.  
For the weight vector $\vec{m} = (-1,1,-1,1)$, the associated 
B-spline is $M(x; -\frac{1}{11},-\frac{1}{35},\frac{1}{36},\frac{1}{34})$; see Figure \ref{Figure:SillyOdd}.
For the weight vector $\vec{m} = (-3,1,1,1)$, the associated B-spline is
$M(x; -\frac{3}{11}, \frac{1}{36}, \frac{1}{35}, \frac{1}{34})$; see Figure \ref{Figure:Silly36}.
Despite their jagged appearances, these spline are continuously differentiable.
\end{example}

    \begin{figure}
    \centering
    \begin{subfigure}{0.475\textwidth}
      \centering
      \includegraphics[width=\textwidth]{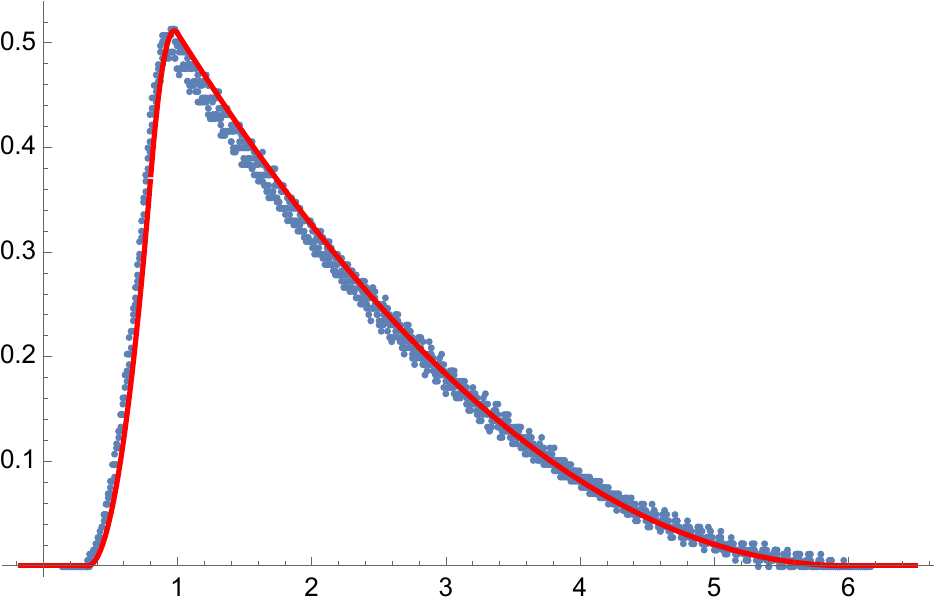}
      \caption{$\vec{m} = (1,1,24,4)$, $\vec{n} = (1,3,4,5)$, $n = 250$}
    \end{subfigure}
    \hfill
    \begin{subfigure}{0.475\textwidth}
      \centering
      \includegraphics[width=\textwidth]{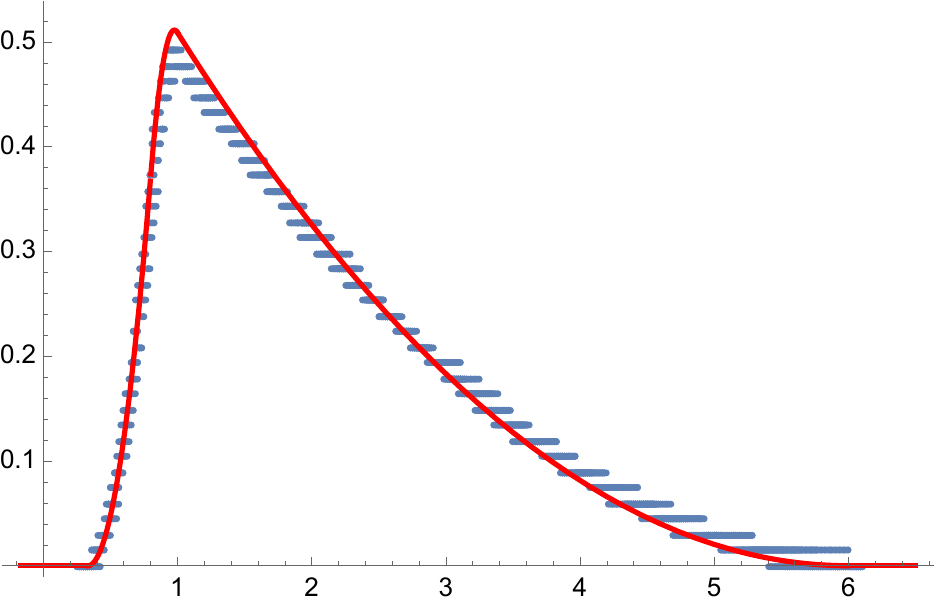}
      \caption{$\vec{m}' = (1,7,6,28)$, $\vec{n}' = (1,21,1,35)$, $n = 500$}
      \end{subfigure}
    \caption{The (non-minimally generated) semigroups $S = \semigroup{1,3,4,5}$ and $S' = \semigroup{1,7,6,28}$ from Example \ref{Example:TwoDifferent}
    yield identical B-splines and hence asymptotically identical weighted factorization-length statistics.}
    \label{Figure:SameSpline}
    \end{figure}

\begin{example}\label{Example:TwoDifferent}
Different semigroups and different weighted factorization lengths may give rise to the same B-spline.
Let
$\vec{n}=(1,3,4,5)$,
$\vec{m}=(1,1,24,4)$, 
$\vec{n}'= (1,21,1,35)$, and
$\vec{m}'= (1,7,6,28)$,
then observe that
\begin{equation*}
\frac{1}{1}=\frac{1}{1}, \qquad 
\frac{3}{1}=\frac{21}{7}, \qquad 
\frac{4}{24}=\frac{1}{6}, \quad \text{and} \quad 
\frac{5}{4}=\frac{35}{28}.
\end{equation*}
Thus, the B-spline $M(x;\frac{1}{3},\frac{4}{5},1,6)$ describes the asymptotic
weighted factorization length statistics in $S = \semigroup{1,3,4,5}$ and in $S' = \semigroup{1,21,1,35}$;
see Figure \ref{Figure:SameSpline}.
\end{example}

\begin{example}
Consider the non-minimally generated numerical semigroup $S = \semigroup{5,5,6,7}$,
for which $\vec{n} = (5,5,6,7)$.  Let $\vec{m} = (1, -2, 3, 1)$ and consider the 
weighted factorization length $\vec{m} \cdot \vec{x}$ on $S$.  The corresponding B-spline is 
\begin{equation*}
M(x; -\tfrac{2}{5},\tfrac{1}{7},\tfrac{1}{5},\tfrac{1}{2})
=
\begin{cases}
 \frac{70}{171} (25 x^2+20 x+4) & \text{if $-\frac{2}{5}\leq x<\frac{1}{7}$}, \\[3pt]
 -\frac{35}{9}  (67 x^2-22 x+1) & \text{if $\frac{1}{7}\leq x<\frac{1}{5}$}, \\[3pt]
 \frac{70}{9} (4 x^2-4 x+1) & \text{if $\frac{1}{5}\leq x\leq \frac{1}{2}$}, \\[3pt]
 0 & \text{otherwise}.
\end{cases}
\end{equation*}
Let $f(x) = e^x \sin(x^2)$ and $[\alpha,\beta]= [\frac{1}{10}, \frac{3}{10}]$.  
One can approximate the sum
\begin{equation*}
\frac{1}{ |\ZZ_S(n)| }\sum_{\substack{\vec{x}\in \ZZ_S(n) \\ \vec{m} \cdot (\vec{x}/n) \in [\alpha,\beta]}}f\left(\frac{\vec{m} \cdot \vec{x}}{n}\right) 
= 0.02334,
\end{equation*}
given to four decimal places, by the integral
$\int_{\alpha}^\beta f(t)M(t; -\tfrac{2}{5},\tfrac{1}{7},\tfrac{1}{5},\tfrac{1}{2})\,dt = 0.02335$.
\end{example}

\begin{figure}
\centering
\begin{subfigure}{0.475\textwidth}
  \centering
  \includegraphics[width=\textwidth]{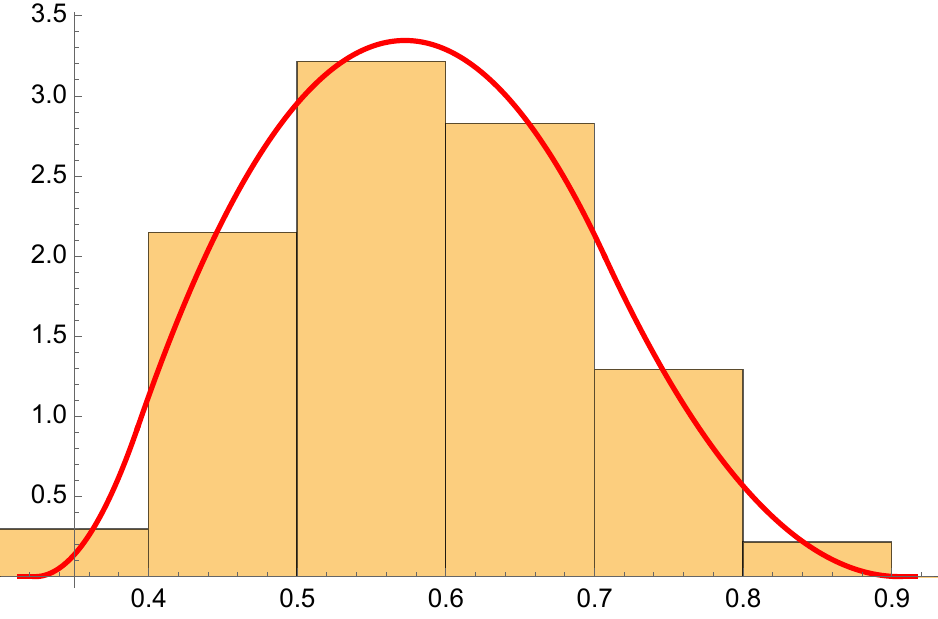}
  \caption{5 bins}
\end{subfigure}
\hfill
\begin{subfigure}{0.475\textwidth}
  \centering
  \includegraphics[width=\textwidth]{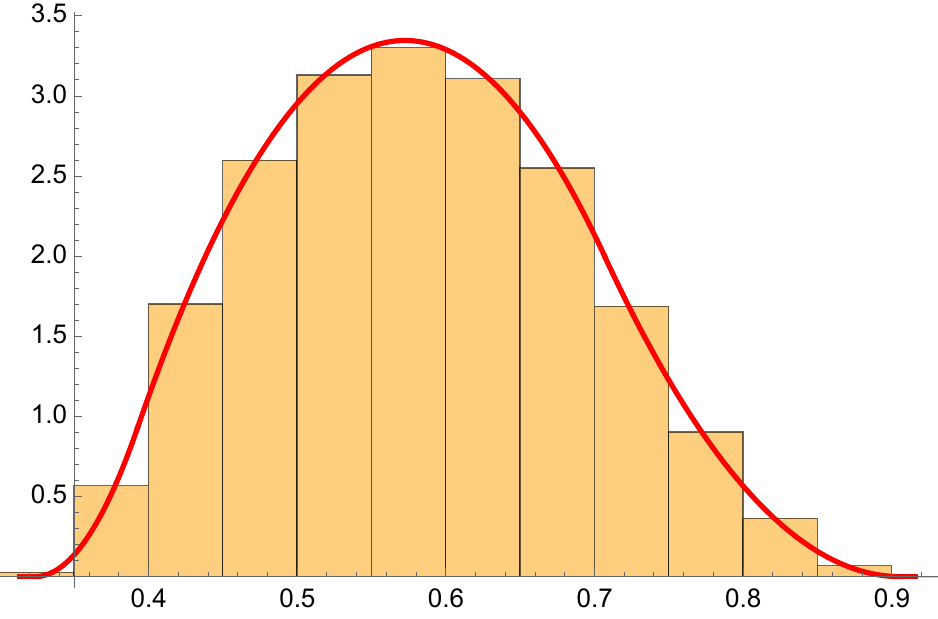}
  \caption{10 bins}
  \end{subfigure}
\\  
\begin{subfigure}{0.475\textwidth}
  \centering
  \includegraphics[width=\textwidth]{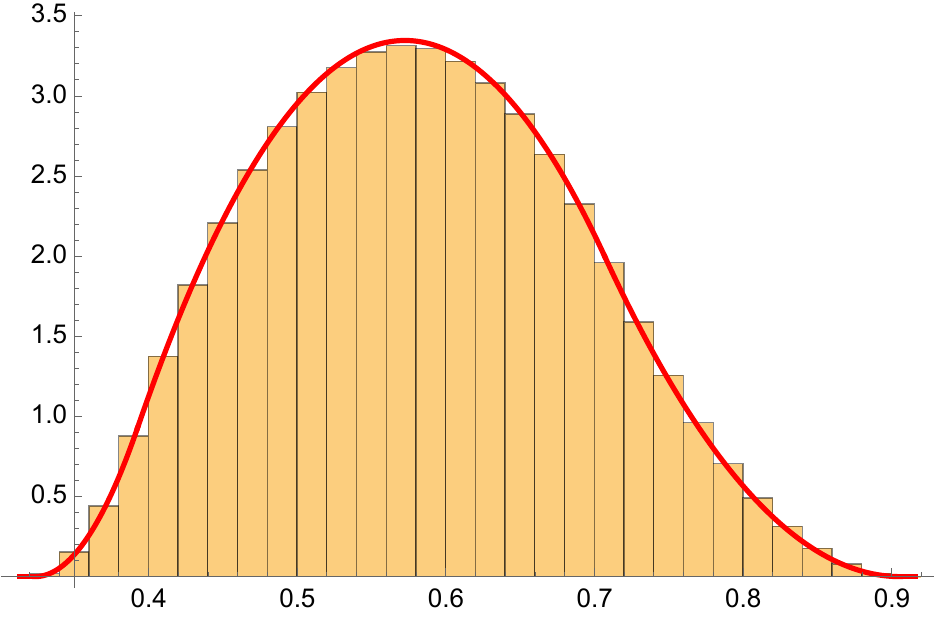}
  \caption{20 bins}
\end{subfigure}
\hfill
\begin{subfigure}{0.475\textwidth}
  \centering
  \includegraphics[width=\textwidth]{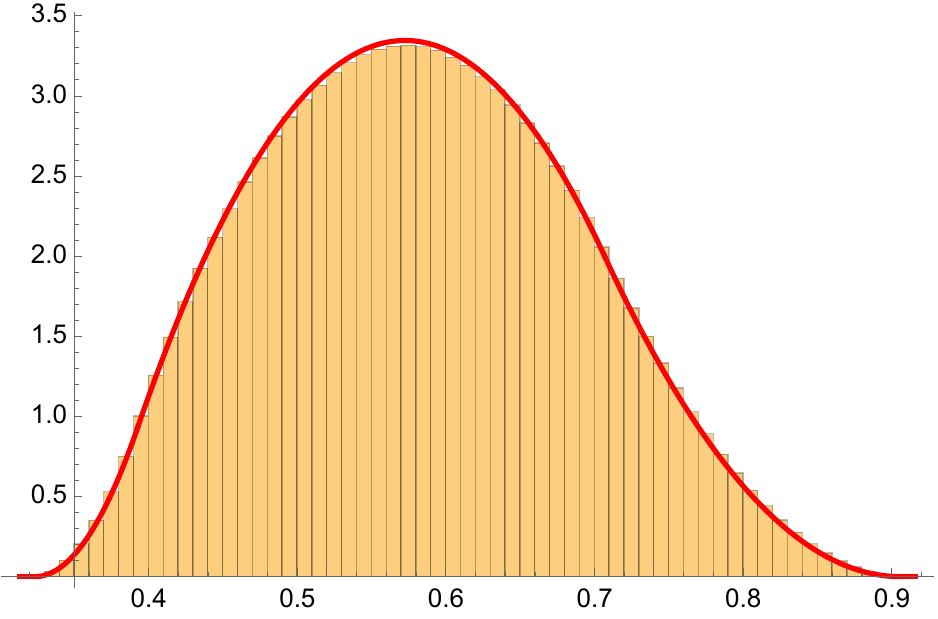}
  \caption{40 bins}
  \end{subfigure}  
\caption{Normalized histograms with various bin sizes for non-integral weighted factorization lengths on the non-minimally generated semigroup $S = \semigroup{2,3,5,8}$ from Example \ref{Example:Real}. Here $\vec{m} = (\sqrt{2},e,\phi ,\pi)$ and $n = 1{,}000$.  }
\label{Figure:Bin}
\end{figure}

\begin{example}\label{Example:Real}
Theorem \ref{Theorem:A} permits non-integer weights.  This requires a different method of visualization
because the weighted factorization lengths need not be integers.  In particular, the multiplicity of weighted factorization lengths is of little use in this setting. We instead create a histogram,
normalized so that its total area is $1$, for the weighted factorization lengths.  For example, let
$S = \semigroup{2,3,5,8}$, a semigroup with a non-minimally generated presentation.
We choose weights $\vec{m} = (\sqrt{2},e,\phi ,\pi )$, so that $\vec{m} \cdot \vec{x} = \sqrt{2}x_1 + e x_2 + \phi x_3 + \pi x_4$.  Here $\phi \approx 1.618$ is the golden ratio.  For large $n$ and an appropriate number of bins, the B-spline $M(x; \frac{1}{\sqrt{2}},\frac{e}{3},\frac{\phi }{5},\frac{\pi }{8} )$ agrees with the envelope of the scaled histogram; see Figure \ref{Figure:Bin}.
\end{example}

\begin{remark}
Because B-splines are probability densities, it makes sense to divide our counts by 
$|\ZZ_S(n)|$, which by \eqref{eq:ZAsymptotic}, is asymptotically equivalent to the expression $n^{k-1} / ((k-1)!n_1n_2\dots n_k)$ that occurs in Theorems \ref{Theorem:A}, and \ref{Theorem:C}; see Figure \ref{Figure:Comparison}. 
\end{remark}

\begin{figure}
\centering
\begin{subfigure}{0.475\textwidth}
  \centering
  \includegraphics[width=\textwidth]{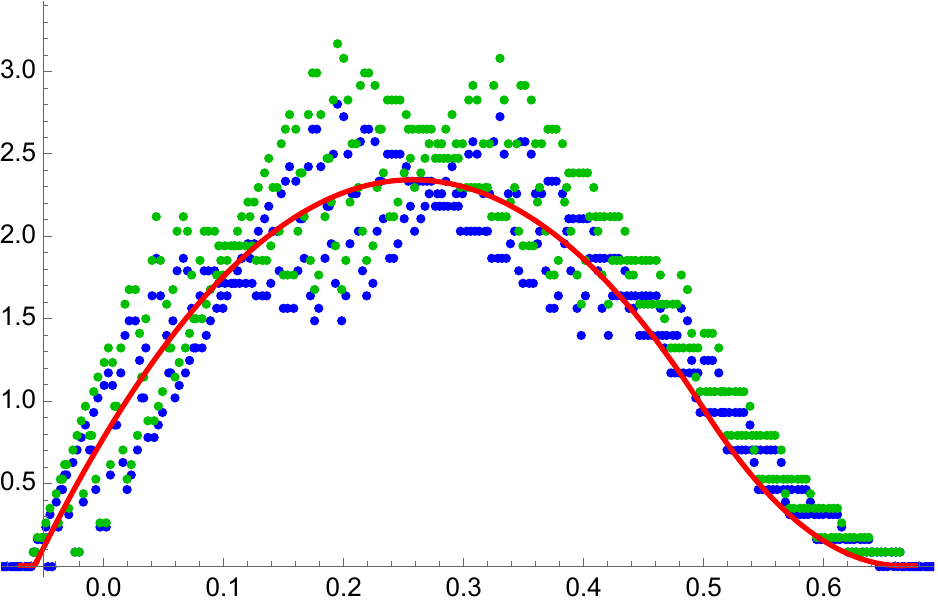}
  \caption{$n=500$}
\end{subfigure}
\hfill
\begin{subfigure}{0.475\textwidth}
  \centering
  \includegraphics[width=\textwidth]{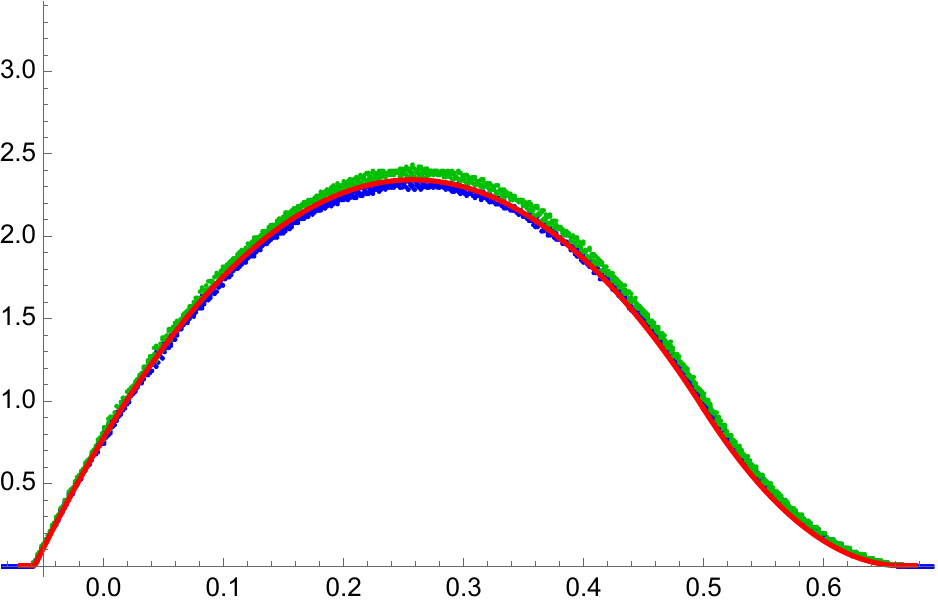}
  \caption{$n=2{,}000$}
\end{subfigure}
\caption{$S = \semigroup{3,4,17,18}$ and $\vec{m} = (2,2,-1,-1)$.
Normalizing by $|\ZZ_S(n)|$ (blue) or 
$\frac{n^{k-1}}{(k-1)! n_1 \cdots n_k}$ (green) yields similar results for large enough $n$.}
\label{Figure:Comparison}
\end{figure}

\begin{example}\label{Example:canes}
    One application of Theorem \ref{Theorem:B} is approximating the number of factorizations of a vector in an affine sub-semigroup of $\Z^2.$ If $S=\langle \vec{a}_1, \vec{a}_2, \ldots, \vec{a}_k\rangle \subseteq \Z^2$ is an affine semigroup, then $\ZZ_S(\vec b)=t_A(\vec b)$, where 
    $A= \fourrowvector{ \vec{a}_1}{ \vec{a}_2}{\ldots}{ \vec{a}_k}$. 
    For example, the \textit{Raising Cane's semigroup}, discussed in \cite{canes}, is the affine semigroup generated by
    $\vec{a}_1 = \twovectorsmall{2}{3}$,
    $\vec{a}_2 = \twovectorsmall{3}{4}$,
    and $\vec{a}_3 = \twovectorsmall{3}{6}$.
The semigroup is named for the popular southern fried chicken restaurant Raising Cane's, where the options for combos include: 3 chicken fingers with two sides, 4 chicken fingers with three sides, and 6 chicken fingers with 3 sides. The number of factorizations of $\vec{b}=\twovectorsmall{m}{n}$ is  the number of ways to order $n$ chicken fingers with $m$ sides. In cases like this where there are three generators, and Theorem \ref{Theorem:B} holds, Lemma 18 from \cite{semigroupsIV} implies that $|t_A(\vec{b})-T_A(\vec{b})|\leq 1$ so the estimates given by $T_A$ are strikingly accurate. Using Theorem \ref{Theorem:B} (or Lemma 18 from \cite{semigroupsIV}), there should be roughly $n/24$ ways to order $n$ chicken fingers and $5n/8$ sides; see Table \ref{Table:canes}.  

The estimates given by $T_A$ become less accurate when there are more generators.
Let $\vec{a}_4=\twovectorsmall{1}{2}$ and $\vec{a}_5=\twovectorsmall{1}{3}$ represent the kids meal with fries and the sandwich combo with fries.  Set $S'=\langle \vec{a}_1,\vec{a}_2,\vec{a}_3, \vec{a}_4,\vec{a}_5\rangle$ and $A'= \fourrowvector{ \vec{a}_1}{ \vec{a}_2}{\ldots}{ \vec{a}_5}$ so that $t_{A'}(\twovectorsmall{m}{n})=|\ZZ_{S'}(\twovectorsmall{m}{n})|$. The estimates $T_{A'}(\vec{b})$ of $t_{A'}(\vec{b})$ shown in Table \ref{Table:canes} are rather less accurate. 
 When we replace $T_{A'}(\vec{b})$ with $\frac{1}{n}|\ZZ_{S''}(n)|M(\tfrac{m}{n}; \tfrac{2}{3},\tfrac{3}{4},\tfrac{3}{6}, \tfrac{1}{2},\tfrac{1}{3})$, where $S''=\langle 3, 4, 6, 2, 3\rangle $, the accuracy improves significantly.
 One can imagine improving the bounds in Theorem \ref{Theorem:B}, but the methods of this paper are limited to bounding the difference between $t_A$ and $T_A$. Thus, it remains unclear how to obtain tight bounds on $\big| |\{\vec{x}\in \ZZ_S(n):\vec{m} \cdot \vec{x}  = m\}|-\frac{1}{n}\ZZ_S(n)M(\tfrac{m}{n};\tfrac{m_1}{n_1}, \tfrac{m_2}{n_2},\dots, \tfrac{m_k}{n_k})\big|$ for general $\vec{n}=[n_i]\in \Z_{>0}^k$, $\vec{m} = [m_i]\in \Z^k$, and $S=\langle n_1,n_2,\ldots, n_k\rangle.$

\end{example}

\begin{table}
\begin{tabular}{c|ccc}
	&							$m=25$ & $m=125$ & $m=625$\\[3pt]
	&                           $n=40$ & $n=200$ & $n=1000$\\
	\hline
	$t_A(\twovectorsmall{m}{n})$& 2          & 8              &  42\\[3pt]
	$T_A(\twovectorsmall{m}{n})$& $ \approx 1.7$ & $\approx 8.3$ & $\approx 41.7$\\[5pt]
	$t_{A'}(\twovectorsmall{m}{n})$& 45 & 2995 & 323169\\[5pt]
	$T_{A'}(\twovectorsmall{m}{n})$& $\approx 19.9$ & $\approx 2488.4$ & $\approx311053.2$\\[5pt]
	$\frac{1}{n}\ZZ_{S''}(n)M(\tfrac{m}{n}; \tfrac{2}{3},\tfrac{3}{4},\tfrac{3}{6}, \tfrac{1}{2},\tfrac{1}{3})$ 	&  $\approx 44.0$ & $\approx 2964.7$ & $\approx 322389.9$\\
\end{tabular}
\caption{Comparison of the partition function and multivariate truncated power for the Raising Cane's semigroup and its extension as in Example \ref{Example:canes}.}
 \label{Table:canes}
\end{table}

\section{Further research}\label{Section:Further}
The results above provide several avenues for further exploration.
For example, one might consider modular versions of our results;
that is, theorems that treat weighted factorization lengths with respect to arbitrary moduli $m$.
Such results were pursued in the factorization-length setting in \cite{Modular}.
However, the methods there do not seem easily generalizable to weighted factorization lengths.  Therefore, it would be of interest to see this line of work come to fruition.

Second, one might pursue tighter bounds in our estimates.  Although their asymptotic 
dependence on $n$ seems relatively tight, at least based upon numerical computations, 
the dependence on $k$, the number
of generators, can likely be improved. This suggests the problem of seeking out extremal examples which maximize the percent error of the approximation given in Theorem \ref{Theorem:B} for fixed value of $k$.  Unfortunately, it is computationally prohibitive to investigate matters when
$k$ is large, unless one uses small $n$, which illuminates nothing. Another avenue of future research is to pursue better approximations by using more information than volumes.

Other directions entirely also beg for further research.  Our work considers additive subsemigroups of $\Z_{\geq 0}$.
However, one might consider $\Z_{\geq 0}^d$ for $d \geq 2$ instead. This direction could likely be achieved by tackling the second extension of Theorem \ref{Theorem:B} suggested below. More generally, one might study finitely generated
abelian semigroups.

Theorem \ref{Theorem:B} might be extended in at least two ways. First, we conjecture that one could remove the hypothesis that $A\Z^k=\Z^2$ and show that when $t_A(\vec b)\neq 0,$ it is approximated by $cT_A(\vec b)$ where $c$ is a constant determined by $A.$ Second, there is the problem of comparing $t_A$ and $T_A$ when $A$ has more than two rows. There is reason to hope that when $A$ is a $k\times (k+1)$ matrix, the error bounds will work out especially nicely.

Lastly, it might be fruitful to apply the results or techniques of this work to situations where a particular partition function (or factorization statistic of a semigroup) is of interest. Partitions of an integer $n$ into at most $k$ parts with largest part at most $j$ can be thought of as nonnegative integer solutions to $x_1+2x_2+\cdots +jx_j=n$ and $x_1+x_2+\cdots +x_k\leq k$. Our theorems can be applied to approximate the number of partitions of this form. These quantities are of current interest due to connections with geometric complexity theory and other areas \cite{Satake}. We thank Aram Bingham for drawing our attention to these possible applications and discussing them at the 2024 Graduate Research Workshop in Combinatorics.  
Another partition function of interest is the Kostant partition function and its $q$-analog introduced by Lusztig \cite{Lusztig}; the $q$-analog of Kostant's partition function associates to each vector the generating function for the distribution of lengths of factorizations of that vector in the semigroup generated by the positive roots of a Lie algebra $\mathfrak{g}$. Limits of these distributions have been studied as the rank of the Lie algebra goes to infinity \cite{qkostant}. The ideas of this paper might be applied to studying limits of these distributions with respect to the input vector.

\bibliography{Splines-arXiv}

\providecommand{\bysame}{\leavevmode\hbox to3em{\hrulefill}\thinspace}
\providecommand{\MR}{\relax\ifhmode\unskip\space\fi MR }
\providecommand{\MRhref}[2]{%
  \href{http://www.ams.org/mathscinet-getitem?mr=#1}{#2}
}
\providecommand{\href}[2]{#2}
\begin{thebibliography}{10}

\bibitem{denum3}
J.~L.~Ramírez Alfonsín, \emph{{Sylvester denumerant}}, {The Diophantine
  Frobenius Problem}, Oxford University Press, 12 2005.

\bibitem{kfrobenius}
Iskander Aliev, Jes\'{u}s~A. De~Loera, and Quentin Louveaux, \emph{Parametric
  polyhedra with at least {$k$} lattice points: their semigroup structure and
  the {$k$}-{F}robenius problem}, Recent trends in combinatorics, IMA Vol.
  Math. Appl., vol. 159, Springer, [Cham], 2016, pp.~753--778. \MR{3526430}

\bibitem{Assi}
Abdallah Assi, Marco D'Anna, and Pedro~A. Garc\'{\i}a-S\'{a}nchez,
  \emph{Numerical semigroups and applications}, RSME Springer Series, vol.~3,
  Springer, Cham, [2020] \copyright 2020, Second edition [of 3558713].
  \MR{4230109}

\bibitem{denum2}
Velleda Baldoni, Nicole Berline, Jes\'{u}s~A. De~Loera, Brandon~E. Dutra,
  Matthias K\"{o}ppe, and Mich\`ele Vergne, \emph{Coefficients of {S}ylvester's
  denumerant}, Integers \textbf{15} (2015), Paper No. A11, 32. \MR{3335388}

\bibitem{denum1}
E.~T. Bell, \emph{Interpolated denumerants and {L}ambert series}, Amer. J.
  Math. \textbf{65} (1943), 382--386. \MR{9043}

\bibitem{BrillGordan}
Daniel Bertrand, \emph{Duality on tori and multiplicative dependence
  relations}, J. Austral. Math. Soc. Ser. A \textbf{62} (1997), no.~2,
  198--216. \MR{1433209}

\bibitem{Billingsley}
Patrick Billingsley, \emph{Probability and measure}, Wiley Series in
  Probability and Statistics, John Wiley \& Sons, Inc., Hoboken, NJ, 2012,
  Anniversary edition [of MR1324786], With a foreword by Steve Lalley and a
  brief biography of Billingsley by Steve Koppes. \MR{2893652}

\bibitem{fancypartitions}
Michel Brion and Mich\`ele Vergne, \emph{Residue formulae, vector partition
  functions and lattice points in rational polytopes}, J. Amer. Math. Soc.
  \textbf{10} (1997), no.~4, 797--833. \MR{1446364}

\bibitem{Butzer}
P.~L. Butzer, M.~Schmidt, and E.~L. Stark, \emph{Observations on the history of
  central {$B$}-splines}, Arch. Hist. Exact Sci. \textbf{39} (1988), no.~2,
  137--156. \MR{975361}

\bibitem{Cassels}
J.~W.~S. Cassels, \emph{An introduction to the geometry of numbers}, Classics
  in Mathematics, Springer-Verlag, Berlin, 1997, Corrected reprint of the 1971
  edition. \MR{1434478}

\bibitem{delta}
S.~T. Chapman, P.~A. Garc\'{\i}a-S\'{a}nchez, D.~Llena, A.~Malyshev, and
  D.~Steinberg, \emph{On the delta set and the {B}etti elements of a
  {BF}-monoid}, Arab. J. Math. (Springer) \textbf{1} (2012), no.~1, 53--61.
  \MR{3040913}

\bibitem{delta1}
S.~T. Chapman, P.~A. Garc\'{\i}a-S\'{a}nchez, D.~Llena, and J.~Marshall,
  \emph{Elements in a numerical semigroup with factorizations of the same
  length}, Canad. Math. Bull. \textbf{54} (2011), no.~1, 39--43. \MR{2797486}

\bibitem{delta5}
Scott~T. Chapman, Felix Gotti, and Roberto Pelayo, \emph{On delta sets and
  their realizable subsets in {K}rull monoids with cyclic class groups},
  Colloq. Math. \textbf{137} (2014), no.~1, 137--146. \MR{3271229}

\bibitem{delta2}
Scott~T. Chapman, Rolf Hoyer, and Nathan Kaplan, \emph{Delta sets of numerical
  monoids are eventually periodic}, Aequationes Math. \textbf{77} (2009),
  no.~3, 273--279. \MR{2520501}

\bibitem{Cones}
Elaine Cohen, T.~Lyche, and R.~F. Riesenfeld, \emph{Cones and recurrence
  relations for simplex splines}, Constr. Approx. \textbf{3} (1987), no.~2,
  131--141. \MR{889550}

\bibitem{comtet}
Louis Comtet, \emph{Advanced combinatorics}, enlarged ed., D. Reidel Publishing
  Co., Dordrecht, 1974, The art of finite and infinite expansions. \MR{460128}

\bibitem{curry1965polya}
Haskell~B Curry and Isaac~J Schoenberg, \emph{On polya frequency functions iv:
  The fundamental spline functions and their limits.}, Tech. report, WISCONSIN
  UNIV MADISON MATHEMATICS RESEARCH CENTER, 1965.

\bibitem{diophantine}
Wolfgang Dahmen and Charles~A. Micchelli, \emph{The number of solutions to
  linear {D}iophantine equations and multivariate splines}, Trans. Amer. Math.
  Soc. \textbf{308} (1988), no.~2, 509--532. \MR{951619}

\bibitem{deBoorRecursion}
Carl de~Boor, \emph{On calculating with {$B$}-splines}, J. Approximation Theory
  \textbf{6} (1972), 50--62. \MR{338617}

\bibitem{whatspline}
\bysame, \emph{What is a multivariate spline?}, I{CIAM} '87: {P}roceedings of
  the {F}irst {I}nternational {C}onference on {I}ndustrial and {A}pplied
  {M}athematics ({P}aris, 1987), SIAM, Philadelphia, PA, 1988, pp.~90--101.
  \MR{976853}

\bibitem{practicalGuide}
\bysame, \emph{A practical guide to splines}, revised ed., Applied Mathematical
  Sciences, vol.~27, Springer-Verlag, New York, 2001. \MR{1900298}

\bibitem{procesiLit}
C.~De~Concini, C.~Procesi, and M.~Vergne, \emph{Vector partition functions and
  index of transversally elliptic operators}, Transform. Groups \textbf{15}
  (2010), no.~4, 775--811. \MR{2753257}

\bibitem{hyperplane}
Corrado De~Concini and Claudio Procesi, \emph{Topics in hyperplane
  arrangements, polytopes and box-splines}, Universitext, Springer, New York,
  2011. \MR{2722776}

\bibitem{MR1380519}
Persi Diaconis and Anil Gangolli, \emph{Rectangular arrays with fixed margins},
  Discrete probability and algorithms ({M}inneapolis, {MN}, 1993), IMA Vol.
  Math. Appl., vol.~72, Springer, New York, 1995, pp.~15--41. \MR{1380519}

\bibitem{Numan}
Marco Donatelli, Carlo Garoni, Carla Manni, Stefano Serra-Capizzano, and
  Hendrik Speleers, \emph{Symbol-based multigrid methods for {G}alerkin
  {B}-spline isogeometric analysis}, SIAM J. Numer. Anal. \textbf{55} (2017),
  no.~1, 31--62. \MR{3592079}

\bibitem{FukshanskyForst}
Maxwell Forst and Lenny Fukshansky, \emph{Counting basis extensions in a
  lattice}, Proc. Amer. Math. Soc. \textbf{150} (2022), no.~8, 3199--3213.
  \MR{4439446}

\bibitem{Forst}
\bysame, \emph{On lattice extensions}, Monatsh. Math. \textbf{203} (2024),
  no.~3, 613--634. \MR{4704784}

\bibitem{Modular}
Stephan~Ramon Garcia, Mohamed Omar, Christopher O'Neill, and Timothy Wesley,
  \emph{Factorization length distribution for affine semigroups {III}: modular
  equidistribution for numerical semigroups with arbitrarily many generators},
  J. Aust. Math. Soc. \textbf{113} (2022), no.~1, 21--35. \MR{4450910}

\bibitem{semigroupsII}
Stephan~Ramon Garcia, Mohamed Omar, Christopher O'Neill, and Samuel Yih,
  \emph{Factorization length distribution for affine semigroups {II}:
  asymptotic behavior for numerical semigroups with arbitrarily many
  generators}, J. Combin. Theory Ser. A \textbf{178} (2021), 105358, 34.
  \MR{4175889}

\bibitem{semigroupsIV}
Stephan~Ramon Garcia, Christopher O'Neill, and Gabe Udell, \emph{Factorization
  length distribution for affine semigroups {IV}: a geometric approach to
  weighted factorization lengths in three-generator numerical semigroups},
  Comm. Algebra \textbf{50} (2022), no.~8, 3481--3497. \MR{4429477}

\bibitem{GOY}
Stephan~Ramon Garcia, Christopher O'Neill, and Samuel Yih, \emph{Factorization
  length distribution for affine semigroups {I}: {N}umerical semigroups with
  three generators}, European J. Combin. \textbf{78} (2019), 190--204.
  \MR{3921068}

\bibitem{delta3}
P.~A. Garc\'{\i}a-S\'{a}nchez, D.~Llena, and A.~Moscariello, \emph{Delta sets
  for symmetric numerical semigroups with embedding dimension three},
  Aequationes Math. \textbf{91} (2017), no.~3, 579--600. \MR{3651565}

\bibitem{delta4}
Pedro~A. Garc\'{\i}a-S\'{a}nchez, David Llena, and Alessio Moscariello,
  \emph{Delta sets for nonsymmetric numerical semigroups with embedding
  dimension three}, Forum Math. \textbf{30} (2018), no.~1, 15--30. \MR{3739324}

\bibitem{Gordan}
P~Gordan, \emph{{\"U}ber den gr{\"o}{\ss}ten gemeinsamen faktor}, Math. Ann.
  \textbf{7} (1873), 443--448.

\bibitem{gruber_lek}
P.~M. Gruber and C.~G. Lekkerkerker, \emph{Geometry of numbers}, second ed.,
  North-Holland Mathematical Library, vol.~37, North-Holland Publishing Co.,
  Amsterdam, 1987. \MR{893813}

\bibitem{KostantSymp}
Victor Guillemin, Eugene Lerman, and Shlomo Sternberg, \emph{Symplectic
  fibrations and multiplicity diagrams}, Cambridge University Press, Cambridge,
  1996. \MR{1414677}

\bibitem{Satake}
Heekyoung Hahn, JiSun Huh, EunSung Lim, and Jaebum Sohn, \emph{From partition
  identities to a combinatorial approach to explicit {S}atake inversion}, Ann.
  Comb. \textbf{22} (2018), no.~3, 543--562. \MR{3845747}

\bibitem{qkostant}
Pamela~E. Harris, Margaret Rahmoeller, and Lisa Schneider, \emph{On the
  asymptotic behavior of the {$q$}-analog of {K}ostant's partition function},
  J. Comb. \textbf{13} (2022), no.~2, 167--199. \MR{4407518}

\bibitem{HatcherBook}
Allen Hatcher, \emph{Algebraic topology}, Cambridge University Press,
  Cambridge, 2002. \MR{1867354}

\bibitem{HeathBrownDiophantine}
D.~R. Heath-Brown, \emph{Diophantine approximation with square-free numbers},
  Math. Z. \textbf{187} (1984), no.~3, 335--344. \MR{757475}

\bibitem{goodRepLit}
G.~J. Heckman, \emph{Projections of orbits and asymptotic behavior of
  multiplicities for compact connected {L}ie groups}, Invent. Math. \textbf{67}
  (1982), no.~2, 333--356. \MR{665160}

\bibitem{Jia}
Rong-Qing Jia, \emph{Symmetric magic squares and multivariate splines}, Linear
  Algebra Appl. \textbf{250} (1997), 69--103. \MR{1420572}

\bibitem{Jung1}
Heinrich Jung, \emph{Ueber die kleinste {K}ugel, die eine r\"{a}umliche {F}igur
  einschliesst}, J. Reine Angew. Math. \textbf{123} (1901), 241--257.
  \MR{1580570}

\bibitem{Jung2}
Heinrich W.~E. Jung, \emph{\"{U}ber den kleinsten {K}reis, der eine ebene
  {F}igur einschlie\ss t}, J. Reine Angew. Math. \textbf{137} (1910), 310--313.
  \MR{1580792}

\bibitem{parmFam}
Franklin Kerstetter and Christopher O'Neill, \emph{On parametrized families of
  numerical semigroups}, Comm. Algebra \textbf{48} (2020), no.~11, 4698--4717.
  \MR{4142067}

\bibitem{interpSplines}
Gary~D. Knott, \emph{Interpolating cubic splines}, Progress in Computer Science
  and Applied Logic, vol.~18, Birkh\"{a}user Boston, Inc., Boston, MA, 2000.
  \MR{1731442}

\bibitem{matroids}
Matthias Lenz, \emph{Splines, lattice points, and arithmetic matroids}, J.
  Algebraic Combin. \textbf{43} (2016), no.~2, 277--324. \MR{3456491}

\bibitem{Lusztig}
George Lusztig, \emph{Singularities, character formulas, and a {$q$}-analog of
  weight multiplicities}, Analysis and topology on singular spaces, {II}, {III}
  ({L}uminy, 1981), Ast\'erisque, vol. 101-102, Soc. Math. France, Paris, 1983,
  pp.~208--229. \MR{737932}

\bibitem{Handbook}
Gheorghe Micula and Sanda Micula, \emph{Handbook of splines}, Mathematics and
  its Applications, vol. 462, Kluwer Academic Publishers, Dordrecht, 1999.
  \MR{1673026}

\bibitem{konstantLitRepTheory}
Marni Mishna, Mercedes Rosas, and Sheila Sundaram, \emph{Vector partition
  functions and {K}ronecker coefficients}, J. Phys. A \textbf{54} (2021),
  no.~20, Paper No. 205204, 29. \MR{4271276}

\bibitem{NathansonParts}
Melvyn~B. Nathanson, \emph{Partitions with parts in a finite set}, Proc. Amer.
  Math. Soc. \textbf{128} (2000), no.~5, 1269--1273. \MR{1705753}

\bibitem{compusimp}
M.~Neamtu and C.~R. Traas, \emph{On computational aspects of simplicial
  splines}, Constr. Approx. \textbf{7} (1991), no.~2, 209--220. \MR{1101063}

\bibitem{ortholattices}
Phong Nguyen and Jacques Stern, \emph{Merkle-{H}ellman revisited: a
  cryptanalysis of the {Q}u-{V}anstone cryptosystem based on group
  factorizations}, Advances in cryptology---{CRYPTO} '97 ({S}anta {B}arbara,
  {CA}, 1997), Lecture Notes in Comput. Sci., vol. 1294, Springer, Berlin,
  1997, pp.~198--212. \MR{1630400}

\bibitem{canes}
Christopher O'Neill and Isabel White, \emph{On minimal presentations of shifted
  affine semigroups with few generators}, Involve \textbf{14} (2021), no.~4,
  617--630. \MR{4332572}

\bibitem{AlfonsinDiophantineFrobenius}
J.~L. Ram\'{\i}rez~Alfons\'{\i}n, \emph{The {D}iophantine {F}robenius problem},
  Oxford Lecture Series in Mathematics and its Applications, vol.~30, Oxford
  University Press, Oxford, 2005. \MR{2260521}

\bibitem{Rosales}
J.~C. Rosales and P.~A. Garc\'{\i}a-S\'{a}nchez, \emph{Numerical semigroups},
  Developments in Mathematics, vol.~20, Springer, New York, 2009. \MR{2549780}

\bibitem{schur1926additiven}
I.~Schur, \emph{Zur additiven zahlentheorie}, Sitzungsberichte der Preussischen
  Akademie der Wissenschaften. Physikalisch-mathematische Klasse, 1926.

\bibitem{StanleyCCA}
Richard~P. Stanley, \emph{Combinatorics and commutative algebra}, second ed.,
  Progress in Mathematics, vol.~41, Birkh\"{a}user Boston, Inc., Boston, MA,
  1996. \MR{1453579}

\bibitem{sylvester}
J.J Sylvester, \emph{On the partition of numbers}, Quart. J. Pure Appl. Math
  \textbf{1} (1857), 141--152.

\bibitem{litonkonstant}
Andr\'{a}s Szenes and Mich\`ele Vergne, \emph{{$[Q,R]=0$} and {K}ostant
  partition functions}, Enseign. Math. \textbf{63} (2017), no.~3-4, 471--516.
  \MR{3852178}

\bibitem{stats}
Grace Wahba, \emph{Spline models for observational data}, CBMS-NSF Regional
  Conference Series in Applied Mathematics, vol.~59, Society for Industrial and
  Applied Mathematics (SIAM), Philadelphia, PA, 1990. \MR{1045442}

\bibitem{truncpowratpoly}
Ren-hong Wang and Zhi-qiang Xu, \emph{Discrete truncated powers and lattice
  points in rational polytope}, Proceedings of the 6th {J}apan-{C}hina {J}oint
  {S}eminar on {N}umerical {M}athematics ({T}sukuba, 2002), vol. 159, 2003,
  pp.~149--159. \MR{2022325}

\bibitem{widmer2}
Martin Widmer, \emph{Counting primitive points of bounded height}, Trans. Amer.
  Math. Soc. \textbf{362} (2010), no.~9, 4793--4829. \MR{2645051}

\bibitem{widmer}
\bysame, \emph{Lipschitz class, narrow class, and counting lattice points},
  Proc. Amer. Math. Soc. \textbf{140} (2012), no.~2, 677--689. \MR{2846337}

\bibitem{explicit}
Zhiqiang Xu, \emph{An explicit formulation for two dimensional vector partition
  functions}, Integer points in polyhedra---geometry, number theory,
  representation theory, algebra, optimization, statistics, Contemp. Math.,
  vol. 452, Amer. Math. Soc., Providence, RI, 2008, pp.~163--178. \MR{2405771}

\end{thebibliography}
\bibliographystyle{amsplain}

\end{document}